\documentclass[a4paper, 10pt, leqno,final]{article}

\usepackage{eucal}	
\usepackage[utf8]{inputenc}
\usepackage{amsmath}
\usepackage{amsfonts}
\usepackage{amsthm}
\usepackage{amssymb}
\usepackage{graphicx}
\usepackage{graphics}
\usepackage{bm}
\usepackage[perpage]{footmisc}

\usepackage{dsfont}
\usepackage{color}
\usepackage{enumerate}
\usepackage[font=footnotesize]{caption}
\usepackage[all]{xy}
\usepackage{xypic}
\allowdisplaybreaks
\usepackage{enumerate}		

\usepackage{mathtools}		% per i simboli := e =:
\usepackage{mathrsfs}		% per il math script
\usepackage{eucal}			% per il mathcal + figo
\usepackage{braket}	
\usepackage[colorlinks=true, linkcolor=black, citecolor=black,
urlcolor=blue]{hyperref}		
\mathtoolsset{showonlyrefs}
\usepackage{mparhack}		% fix bug for \marginpar
\usepackage{relsize}			% dimensioni testo
\usepackage{a4wide}		% aumenta i margini di pagina
\usepackage{booktabs}		% per fare bene le tabelle
\usepackage{multirow}		% per celle su piÂ righe nelle tabelle
\usepackage{caption}		% per le didascalie di tabelle e figure
\usepackage{rotating}
\usepackage{varioref}		
\usepackage{footmisc}

\newcommand{\trasp}[1]{{#1}^\mathsf{T}}

\numberwithin{equation}{section}
\theoremstyle{plain}
\newtheorem{thm}{Theorem}[section]
\newtheorem{cor}[thm]{Corollary}
\newtheorem{lem}[thm]{Lemma}
\newtheorem{prop}[thm]{Proposition}
\theoremstyle{definition}

\newtheorem{dfn}[thm]{Definition}

\theoremstyle{remark}
\newtheorem{rmk}[thm]{Remark}
		\newtheorem{rem}[thm]{Remark}		
				\newtheorem{note}[thm]{Notation}	
\renewcommand{\=}{\coloneqq}			% definisce :=
\newcommand{\N}{\mathds{N}}

\newcommand{\Q}{\mathbb{Q}}

\newcommand{\R}{\mathds{R}}

\newcommand{\OO}{\mathrm{O}}
\newcommand{\C}{\mathds{C}}

\newcommand{\diff}{\mathrm{d}}

\newcommand{\SO}{\mathrm{SO}}
\newcommand{\ccc}{\mathrm c}

\newcommand{\Id}{I}

\newcommand{\XX}{\mathbb X}

\newcommand{\mscal}[2]{\langle#1,#2\rangle_{M}}
\newcommand{\mnorm}[1]{|#1|_M}

\DeclareMathOperator{\diag}{diag}		% diagonal matrix

\newcommand{\cc}{\mathrm{CC}}
\newcommand{\cccc}{\mathrm{CCC}}
\newcommand{\sbc}{\mathrm{SBC}}
\newcommand{\csbc}{\mathrm{CSBC}}
\newcommand{\nullity}[1]{\iota^0(#1)}		% coindice di Morse
\newcommand{\iMor}[1]{\iota^{-}(#1)}

 \newcommand{\coiMor}[1]{\iota^{+}(#1)}

\renewcommand{\=}{\coloneqq}			% definisce :=
\usepackage{bm}

\title{Morse theory for $S$-balanced configurations\\ in the Newtonian  $n$-body problem}
\author{Luca Asselle, Alessandro Portaluri}
\date{\today}

\date{\today}
\begin{document}
 \maketitle

\begin{abstract}
For the Newtonian (gravitational) $n$-body problem in the Euclidean $d$-dimensional space, the simplest possible solutions are provided by those rigid motions (homographic solutions) in which each body moves along a Keplerian orbit and the configuration of the $n$-body is a (constant up to rotations and scalings) \textit{central configuration}. For $d\le 3$, the only possible homographic motions are 
those given by central configurations. For $d \ge 4$ instead, new possibilities arise due to the higher complexity of the orthogonal group $\OO(d)$, as observed by Albouy and Chenciner in \cite{AC98}. For instance, in $\R^4$  it is possible to rotate in two mutually orthogonal planes with different angular velocities. This produces a new balance between gravitational forces and centrifugal forces providing new periodic and quasi-periodic motions. So, for $d\ge 4$ there is a wider class of  {\em $S$-balanced configurations\/} (containing the central ones) providing simple solutions of the $n$-body problem, which  
can be characterized as well through critical point theory. 

In this paper, we first provide a lower bound on the number of balanced (non-central) configurations in $\R^d$, for arbitrary $d\ge 4$, and establish 
a version of the $45^\circ$-theorem for balanced configurations, thus answering some of the questions raised in \cite{Moe14}. Also, a careful study of the asymptotics of the coefficients of the 
Poincar\'e polynomial of the collision free configuration sphere will enable us to derive some rather unexpected qualitative consequences on the count of $S$-balanced configurations. 
In the last part of the paper, we focus on the case $d=4$ and provide a lower bound on the number of periodic and quasi-periodic motions of the gravitational $n$-body problem which improves a previous celebrated result of McCord \cite{McCord96}. 

\vspace{2mm}

\end{abstract}

%
%%%%%%%%%%%%%%%%%%%%%%%%%%%%%%%%%%%%%%%%%%%%%%%%%%%%%%%%%%%%%%%%%%
%%
%%
%%
%%
%%
%%
%%
%%%%%%%%%%%%%%%%%%%%%%%%%%%%%%%%%%%%%%%%%%%%%%%%%%%%%%%%%%%%%%%%%%

\tableofcontents

\section{Introduction}\label{sec:introduction}

 The Newtonian $n$-body problem concerns the motion of $n$ point particles with masses $m_j \in \R^+$ and position $q_j \in \R^d$, where $j=1, \ldots, n$ and $ d \ge 2$, interacting each other according to Newton's law of inverses squares. The equations of motion read 
 \begin{equation}\label{eq:motions}
 m_j \ddot q_j = \dfrac{\partial U}{\partial q_j}, \quad \textrm{ where } \quad U(q_1, \ldots, q_n)\=\sum_{i<j}\dfrac{m_i m_j}{|q_i-q_j|}. 	
 \end{equation}
 As the center of mass has an inertial motion, there is no loss in generalities in assuming that the center of mass lies at the origin. 
 The {\sc configuration space with center of mass at the origin\/} is defined as 
 $$ \mathbb X\= \Set{(q_1, \ldots, q_n) \in \R^{dn}| \sum_{i=1}^n m_i q_i=0}.$$  
The potential $U$ is not defined when (at least) two particles {\em collide\/}, meaning that they have the same position. Therefore, we define the space of {\sc collision free configurations\/} as   
\[
\widehat {\mathbb X}\=\Set{q=(q_1, \ldots, q_n) \in \mathbb X| q_i \neq q_j \ \textrm{ for } i \neq j}= \mathbb X\setminus \Delta,
\]
where  
$$\Delta \= \Set{q=(q_1, \ldots, q_n) \in \R^{dn}|q_i = q_j \ \textrm{ for } i \neq j}$$ 
is the  {\sc collision set\/}. Let $M$ be the $(nd \times nd)$-diagonal {\sc mass matrix\/} defined by 
$$M\= \diag(\underbrace{m_1, \ldots, m_1}_{d{\text{-times}}}, \ldots, \underbrace{m_n, \ldots, m_n}_{d{\text{-times}}})$$ 
and let  
\[
\mscal{\cdot}{\cdot}\= \langle M \cdot, \cdot \rangle \quad \textrm{ and } \quad \mnorm{\cdot}\= \langle M \cdot, \cdot \rangle^{1/2}
\]
be respectively the {\sc mass scalar product\/} and the {\sc mass norm,\/} where $\langle \cdot, \cdot \rangle$ denotes the Euclidean product in $\R^{nd}$. 
In this notation, the equations of motion can be written as
\begin{equation}\label{eq:Newton-intro}
	\ddot q= M^{-1}\nabla U(q).
\end{equation}

Among all possible configurations of the system, a crucial role in penetrating the intricate dynamics of this problem, is played by the so-called {\sc central configurations\/} ($\cc$ in shorthand notation), namely configurations in which $M^{-1}\nabla U(q)$ is parallel to $q$: 
\begin{equation}\label{eq:cc-eq-2}
	M^{-1}\nabla U(q) + \lambda q=0
\end{equation}
%These  configurations are a special type of {\sc $S$-balanced configurations\/}.    
In other words, the acceleration vector of each particle is pointing towards the origin with magnitude proportional to the distance to the origin. As a straightforward consequence of the homogeneity of the potential we obtain that the proportionality constant $\lambda$ is actually equal to $-U(q)/\mnorm{q}^2$. 

 Equation~\eqref{eq:cc-eq-2} is a nonlinear algebraic equation which turns out to be extremely hard to solve. Despite substantial progresses (starting from the work of - among others - Smale, Conley, Albouy, Chenciner, McCord, Moeckel, Pacella) have been made in the last decades, many basic questions about $\cc$ still remain unsolved.
Nevertheless, there are several reasons why $\cc$ are of interest in the $n$-body problem and more generally in Celestial Mechanics: 
 \begin{itemize}
\item[-] Every $\cc$ defines a {\sc homothetic solution\/} of~\eqref{eq:Newton-intro}, namely a solution which preserves its shape for all time while receding from or collapsing into the center of mass.
\item[-] Planar $\cc$ give rise to a family of periodic motions of~\eqref{eq:Newton-intro}, the so-called {\sc relative equilibria\/}, in which each of the bodies moves on a circular Kepler orbit. In other words, the configuration rigidly rotates at a constant angular speed about the center of mass. These motions are true equilibrium solutions in a uniformly rotating coordinate system and correspond to elliptical orbits of the Kepler problem of eccentricity $e=0$. It is worth observing that relative equilibria are the simplest periodic solutions of the $n$-body problem having an explicit formula.
\item[-] Planar $\cc$ more generally give rise to a family of  {\sc homographic  solutions\/} of~\eqref{eq:Newton-intro} in which each particle traverses an elliptical orbit  with eccentricity $ e \in (0,1)$. In this case, both the radius $r(t)$ of the orbit and the angular speed $\dot \vartheta(t)$ are time-dependent solutions of the Kepler problem (in polar coordinates). Moreover, the configuration remains similar to the initial configuration throughout the motion, varying only in size. 
\item[-] $\cc$ control the qualitative behavior of an important class of colliding solutions (and parabolic motions) of the $n$-body problem. More precisely, if an orbit approaches or recedes from total collapse,  it does so approximating homothetic orbits. So, orbits passing arbitrarily close to a total collision do it by first approaching one homothetic motion and then receding from the collision following a (possibly different) homothetic motion.
 \end{itemize}
 
 In other words, for the $n$-body problem in $\R^d$, $d \le 3$, any $\cc$ generates a homothetic ejection or collapse. Morevorer, any planar $\cc$ generates planar Keplerian homographic motions, whereas
 spatial non-planar $\cc$ can only produce homothetic motions. Configurations which are not central cannot produce homographic motions at all. If we instead allow dimensions $d \ge 4$, then - as observed by Albouy and Chenciner in \cite{AC98} (cfr. also \cite{Moe14}) - there is a wider class of the so-called ``$S$-balanced configurations'' which produces relative equilibria of the $n$-body problem. 
 These new phenomena are due to the higher complexity of the group of rotations in higher dimensions, which allows, for example, to rotate in two mutually orthogonal planes with different angular velocities. This leads to new ways of balancing the gravitational forces with centrifugal forces in order to get new relative equilibria of the $n$-body problem. We shall notice that, in contrast with the case $d=2$, such relative equilibria need not be periodic in time. Indeed, if the angular velocities are rationally independent, then the resulting motions will be only quasi-periodic.  
 
In order to formally define such a class of configurations, we consider positive real numbers $s_1\ge  \ldots\ge  s_d > 0$  (possibly not all different) and set 
\[
\widehat S := \diag(\underbrace{S,...,S}_{n\text{-times}}) \in \R^{nd\times nd},
\]
where $S= \diag (s_1,...,s_d)$. An {\sc $S$-balanced configuration\/}, ($\sbc$ in shorthand notation),  is an arrangement of the   masses whose associated configuration vector $q\in \widehat {\mathbb X}$  satisfies
	\begin{equation}\label{eq:intro-s-balanced-conf}
		\nabla U(q)+ \lambda \widehat SM q=0
	\end{equation} 
for some real (positive) constant $\lambda$. We observe that for $n=3$  a planar non-equilateral and non-collinear isosceles triangle is a $\sbc$ (but not a $\cc$) as soon as the two symmetric masses are equal. Also, the class of $\sbc$ strictly includes $\cc$ (corresponding, in fact, to $S= \Id$), and Equation~\eqref{eq:intro-s-balanced-conf} is in general only invariant under the (diagonal) action of a subgroup of $\SO(d)$, which is proper as soon as $S\neq \Id$. In the extremal case in which the (diagonal) entries of $S$ are pairwise distinct, the symmetry group of Equation~\eqref{eq:intro-s-balanced-conf} is trivial.

\begin{rmk}
The usual definition of $S$-balanced configurations (cfr. e.g. \cite{Moe14}) requires $S$ to be a symmetric positive definite matrix. However,  
since the problem is invariant by unitary transformation, there is no loss of generality in assuming that $S$ be in diagonal form, i.e. 
that the spectral basis coincide with the canonical basis of $\R^d$. Also, one requires $S$ to be minus the square of a skew-symmetric matrix (e.g. a complex structure), 
and hence, in particular, that all eigenvalues of $S$ have even multiplicity.
As we shall see later, it will be very convenient to allow $S$ to have eigenvalues of odd multiplicity (in particular simple). \qed
\end{rmk}

It is worth pointing out that, in even dimensions, if $S=-A^2$ where $A$ is an antisymmetric matrix, every $\sbc$ gives rise to a uniformly rotating relative equilibrium solution of Equation~\eqref{eq:Newton-intro}. To get more general homographic solutions it is still necessary to have a $\cc$ in which the bodies run on planar Keplerian ellipses. However, ellipses corresponding to different bodies 
may lie in different planes. 

In order to find solutions of~\eqref{eq:intro-s-balanced-conf}, we will exploit the fact that $\sbc$ (as well as $\cc$) admit a variational characterization as critical points of the restriction of $U$ to the 
{\sc collision free configuration sphere}
$$\widehat{\mathbb S}\= \mathbb S \setminus \Delta, \qquad \text{where}\qquad  \mathbb S \=\Set{q \in \mathbb X| I_S(q)=1}$$ 
is the {\sc configuration sphere}  and  $I_S(q)$ if the {\sc S-weighted moment of inertia}, namely the norm squared associated to the scalar product induced by $ \widehat SM$.
The proof of Shub's lemma carries over word by word to this more general situation, thus showing that $\sbc$ cannot accumulate on the collision set $\Delta$. In particular, we can
 find a neighborhood of $\Delta$ in $\mathbb S$ which contains no $\sbc$. This opens up the possibility of using Morse theoretical methods to obtain lower bounds on the number 
 of solutions to~\eqref{eq:intro-s-balanced-conf}. However, in doing this, one has to keep in mind that:
 \begin{itemize}
 \item Equation~\eqref{eq:intro-s-balanced-conf} has a non-trivial symmetry group whenever 
 at least two entries of the matrix $S$ are equal, and hence critical points are \textit{not} isolated unless all entries of $S$ are distinct. Also, such an action is not free (actually, non even locally free) as soon as $d\ge 3$. Therefore, the quotient space is in general not even an orbifold but rather an Alexandrov space, and hence a delicate stratified Morse theory is needed when working on the quotient space. Methods such as Morse-Bott theory, which allow to deal with functions having (Morse-Bott) 
 non-degenerate critical manifolds, do not take into account the symmetries of the problem and hence necessarily lead to weaker results.  
 \item Even if one is able to give lower bounds on the number of $\sbc$, for instance via equivariant Morse theory or Morse-Bott theory, one has additionally to show that the $\sbc$ that one finds are actually not $\cc$.   This is the same problem that one has to face when trying to exclude  that spatial $\cc$ are actually not planar: despite the lower bounds provided by Pacella \cite{Pac86}, it is so far not known whether there is more than one spatial  $\cc$ which is not planar. 
 \item All previously mentioned methods are essentially based on equivariant homology and intersection homology theories, respectively, which are
 computationally inaccessible in full generality. 
 \end{itemize}
 
 We will circumvent all these difficulties at once by restricting our attention to a subspace of $\R^d$ in which the eigenvalues of $S$ are pairwise distinct
 (in what follows this will be called the ``reduction to Assumption {\bf(H1)}'' argument). Assumption {\bf (H1)} will {\em kill\/} all symmetries of Equation~\eqref{eq:intro-s-balanced-conf}, thus allowing us to apply standard Morse theory arguments. Also, the fact that all entries of $S$ are distinct implies that solutions of~\eqref{eq:intro-s-balanced-conf}
 which are not collinear cannot be $\cc$. More will be said about Assumption {\bf(H1)} in Section~\ref{sec:inertia-balanced}. Therefore, 
 applying classical Morse theory to $\widehat U:=U|_{\widehat {\mathbb S}}$ under Assumption {\bf(H1)}, we derive lower bounds on the number of $\sbc$ assuming non-degeneracy. These lower bounds will depend on suitable spectral gap conditions involving the eigenvalues of the matrix $S$ 
 (cfr. Theorems~\ref{thm:main1}, \ref{thm:main2}, and \ref{thm:main3} for the precise statements).  A first key step in this direction is provided by a careful investigation of  collinear $\sbc$ (in shorthand notation $\csbc$) and their inertia indices  (cfr.  Lemma~\ref{thm:morse-index-co-balanced} for further details). The Poincar\'e polynomial of $\mathbb S$ is the same as in the case of $\cc$, however establishing 
precise growth estimates on its coefficients (cfr. Proposition~\ref{lem:coeffofpoincarepolynomial}) we are able to derive non-trivial and rather unexpected qualitative consequences on the count of $\sbc$: roughly speaking, $\csbc$ corresponding to smaller eigenvalues of $S$ will in many cases contribute more 
 than $\csbc$ corresponding to the largest eigenvalue of $S$ (this should be compared with \cite[Pag. 151]{Moe14}). 
We stress the fact that all lower bounds obtained via the reduction to {\bf(H1)} argument provide genuine $\sbc$.

In the last section we prove a sharper lower bound on the number of non-degenerate $\sbc$ in dimension $d=4$ for $S= \diag(s,s,1,1)$. Even if the proof is again based on Morse theory, the argument is completely different than the one used in the proof of Theorem~\ref{thm:main1}. After restricting our attention to the $\{0\}\times \R^3\subset \R^4$, we
quotient out $\widehat{\mathbb S}$ by the $\SO(2)$-action given by rotations in the $\{0\}\times \R^2$-plane after removing the manifold $\mathbb Y$ of 
collinear configurations contained in $\R\times \{0\}$. 
This is done by first establishing a so-called {\sc $45^\circ$ Theorem for $\sbc$\/} (cfr. Theorem~\ref{thm_45-theorem}). 
Therefore, the quotient
\[
\overline{\mathcal S}= (\widehat{\mathbb S}\setminus \mathbb Y)/\SO(2)
\]
is a manifold. We then compute in Theorem~\ref{thm:homology-of-sbar} the homology of $\overline{\mathcal S}$ by using Alexander duality and the Gysin long exact sequence, 
and finally obtain the desired lower bound by explicitly computing the sum of the Betti numbers of $\overline {\mathcal S}$.
This  has an important consequence on the lower bound of periodic orbits of the gravitational $n$-body problem. In fact, assuming non-degeneracy, we have at least 
\begin{equation}\label{eq:newlowerbound-intro}
	 n! \Big ( h(n) + \dfrac 12 + \dfrac 1n\Big ), \qquad  h(n) := \sum_{j=3}^{n} \frac 1j,
\end{equation}
relative equilibria in $\R^4$ for fixed $s>1$. Such a number is roughly twice as large as the lower bound proved by McCord in \cite{McCord96}:
among all of these relative equilibria, at least 
\[
\dfrac{n!}{2} \big (h(n) + 1)
\]
are produced by $\cc$, namely $n! h(n)/2$ corresponding to planar non-collinear $\cc$ (McCord's estimate) and $n!/2$ corresponding to collinear $\cc$ ($\cccc$ in shorthand notation). By subtracting the latter integer to the former one, we get at least
\begin{equation}\label{eq:finale-intro}
n! \Big (\dfrac{h(n)}{2} + \dfrac 1n \Big )
\end{equation}
relative equilibria which do not come from McCord's estimate, but could anyway come from $\cc$ (unlike in Theorem~\ref{thm:main1}, this cannot be excluded a priori). 
Finally, we observe that the relative equilibria coming from genuine $\sbc$ are periodic in time for $s\in \Q$, and quasi-periodic otherwise.

%It is a challenging longstanding open problem to establish a Morse theory for periodic orbits in the case of gravitational or weakly self-interacting potential. Since, relative equilibria are special periodic solutions of the gravitational $n$-body problem, we conclude that in $R^4$ we have at least 
%\[
 %n! \Big ( h(n) + \dfrac 12 + \dfrac 1n\Big )
%\]
%periodic orbits.  By choosing in the matrix $S$,  the eigenvalue  $s\in \R\setminus \Q$ 
%we produce  quasi-periodic motions. So,  the number  given at Equation~\eqref{eq:newlowerbound-intro} provided a lower bounds on the number of {\sc quasi-periodic motions\/} of the $n$-body problem in $\R^4$. 

We finish this introduction by summarizing, for the reader's sake, some of the notation that we shall use henceforth throughout the paper.

\begin{itemize}
\item $M$ is the $nd$-dimensional mass matrix; $\Id_d$ is the $d$-dimensional identity matrix; $S$ is a $d$-dimensional diagonal matrix, $\widehat S \= \diag(S,\ldots,S)$.
\item $\mscal{\cdot}{\cdot}, \mnorm{\cdot}$ denote respectively the mass scalar product and the mass norm.
 \item $I(q), I_S(q)$ denote the moment of inertia and the $S$-weighted moment of inertia (w.r.t. origin). 
\item $\XX$ is the configuration space with center of mass at the origin; $\Delta$ is the collision set and $\widehat {\mathbb X}\=X\setminus \Delta$ is the collision free configuration space; $
 \mathbb S$ is the configuration sphere and $\widehat{\mathbb S}\= \mathbb S \setminus \Delta$ is the collision free configuration sphere.
 \item $\widehat U\= U|_{\widehat{\mathbb S}}$ is the restriction of $U$ to $\widehat{\mathbb S}$.
\item $\cc, \sbc, \cccc, \csbc$ denote respectively the set of central configurations, $S$-balanced configurations, collinear $\cc$ and finally collinear $\sbc$.
\item $\iota^-(q), \iota^+(q),\iota^0(q)$ denote respectively the Morse index, the Morse co-index, and the nullity of the $\sbc$ $q$ as a critical point of $\widehat U$. 
\end{itemize}

\noindent \textbf{Acknowledgments.}  We warmly thank the anonymous referee for his careful reading of the paper and for his precious comments and suggestions which enabled us to improve significantly the paper and remove several inaccuracies.
Luca Asselle is partially supported by the DFG-grant 380257369 ``Morse theoretical methods in Hamiltonian dynamics''.

%%%%%%%%%%%%%%%%%%%%%%%%%%%%%%%%%%%%%%%%%%%%%%%%%%%%%%%%%%%%%%%%%%
%%
%%
%%
%%
%%
%%
%%
%%%%%%%%%%%%%%%%%%%%%%%%%%%%%%%%%%%%%%%%%%%%%%%%%%%%%%%%%%%%%%%%%%
\section{$S$-balanced  configurations: definition and basic properties}\label{sec:inertia-balanced}

This section is devoted to recall the definition of $S$-balanced configurations and to collect some of their properties. Our basic reference, for this section, is \cite{Moe14}. 

For $n \ge 2$, let $m_1, \ldots, m_n$ be positive real numbers (which can be thought of as the masses of the $n$ points) and let $M$ be the diagonal (block) $(nd\times nd)$-matrix defined as 
\[
M\=[M_{ij}]_{i,j=1}^n,\quad M_{ij}\= m_j \delta_{ij}\Id_d,
\]
where $\Id_d$ denotes the $d$-dimensional identity matrix, $ d \ge 1$. Denoting by $\langle \cdot, \cdot \rangle$ the Euclidean product in $\R^{nd}$, we denote by 
\[
\mscal{\cdot}{\cdot}\= \langle M \cdot, \cdot \rangle \quad \textrm{ and } \quad \mnorm{\cdot}\= \langle M \cdot, \cdot \rangle^{1/2}
\]
respectively the {\sc mass scalar product\/} and the {\sc mass norm.\/} 
 As the center of mass has an inertial motion, it is not restrictive to prescribe its position at the origin. Therefore, we define the {\sc configuration space with center of mass at the origin\/} as
 \[
 \mathbb X\= \Set{(q_1, \ldots, q_n) \in \R^{dn}| \sum_{i=1}^n m_i q_i=0}.
 \]
 It is readily seen that $\mathbb X$ is an $N$-dimensional (real) vector space where $N\= d(n-1)$. We define the space of {\sc collision free configurations\/} as   
\[
\widehat {\mathbb X}\=\Set{q=(q_1, \ldots, q_n) \in \mathbb X| q_i \neq q_j \ \textrm{ for } i \neq j}= \mathbb X\setminus \Delta,
\]
where  
$$\Delta \= \Set{q=(q_1, \ldots, q_n) \in \R^{dn}|q_i = q_j \ \textrm{ for } i \neq j}$$ 
denotes  the  {\sc collision set\/}.   
The   {\sc Newtonian potential\/}  $U:\widehat {\mathbb X}\to \R$ at the {\sc  configuration vector\/} $q$ is given by
\[
U(q)\=\sum_{\substack{i <j} } \dfrac{m_i m_j}{|q_i-q_j|}.
\]
and so, Newton's equations of motion read as follows
\[
M \ddot q= \nabla U(q) = \sum_{i\neq j} \frac{m_im_j}{|q_i-q_j|^3} \cdot (q_i-q_j), 
\]
where $\nabla U$ denotes the $nd$-dimensional gradient.
Given positive real numbers $s_1\ge \ldots\ge s_d > 0$  (possibly not all different), we let $S$ be the diagonal $(d\times d)$-matrix defined by  
$S= \diag (s_1,...,s_d)$, and 
\[
\widehat S := \diag(\underbrace{S,...,S}_{n\text{-times}}) \in \R^{nd\times nd}.
\]
\begin{dfn}\label{def:s-balanced-conf}
	An {\em $S$-balanced configuration\/}, ($\sbc$ in shorthand notation),  is an arrangement of the   masses whose associated configuration vector $q\in \widehat {\mathbb X}$  satisfies
	\begin{equation}\label{eq:s-balanced-conf}
		\nabla U(q)+ \lambda \widehat SM q=0
	\end{equation} 
	for some real (positive) constant $\lambda$.
\end{dfn}

\begin{rem}
In \cite{Moe14} the matrix $S$ is assumed to be symmetric and positive definite. Nevertheless, we are not losing generality by assuming here that 
$S$ is in diagonal form. Indeed, a symmetric matrix can always be diagonalized by choosing a spectral basis. In other words, here we are simply assuming that the spectral basis coincides with the standard basis of $\R^d$. This is possible since the formulation of the problem doesn't depend upon this choice. 
Also, one usually assumes that $S$ be minus the square of a skew-symmetric matrix, hence in particular that each eigenvalue of $S$ have even multiplicity. As we shall see later, it will be very convenient to allow $S$ to have eigenvalues of odd multiplicity, in particular simple eigenvalues. \qed
\end{rem}

The two extreme cases correspond to $s_1=\ldots= s_d=1$, respectively to $s_i \neq s_j$ for all $i \neq j$. In the former case we get the well-known notion of {\sc central configuration\/} ($\cc$ for short) for which a rich literature is nowadays available, whereas the latter case will be the main object of interest of the present paper. 
\begin{dfn}\label{def:inertia-momentum}
The  {\sc moment of inertia (w.r.t. $O$)\/} for the  configuration vector $q$ is defined by  
\[
I(q)\=\mnorm{q}^2.
\]
Analogously, we term {\sc $S$-weighted moment of inertia (w.r.t. $O$)\/}
the positive number 
\[
I_S(q)\=\langle \widehat S Mq,  q\rangle=|q|^2_{S}
\]
 where $|\cdot|^2_{S}$ denotes the norm squared induced by the scalar product $\langle \cdot, \cdot \rangle_{S}\= \langle \widehat SM \cdot, \cdot \rangle$.
\end{dfn}
\begin{rem}
Taking the scalar product of Equation~\eqref{eq:s-balanced-conf} with $q$ and using Euler's theorem for positively homogeneous functions,  we get that the (positive) constant $\lambda$ appearing in Equation~\eqref{eq:s-balanced-conf} is given by 
\[
\lambda\= \dfrac{U(q)}{I_S(q)}.
\]
In particular, for any $\sbc$ $q$, we get a continuous family of $\sbc$ by scaling.   \qed
\end{rem}
\begin{dfn}\label{def:conf-mnfld}
Under the previous notation, we define the {\sc  configuration sphere\/} and the {\sc collision free configuration sphere\/} as follows 
\[
 \mathbb S \=\Set{q \in \mathbb X| I_S(q)=1}\quad \textrm{ and } \quad \widehat{\mathbb S}\= \mathbb S \setminus \Delta 
\]
We refer to its elements as {\sc (normalized) configuration vectors.\/}
\end{dfn}
It is immediate to check that  $ \mathbb S $ is a smooth compact manifold diffeomorphic to a  $(N-1)$-dimensional sphere. As a direct consequence of the scaling property of Equation~\eqref{eq:s-balanced-conf}, with any $\sbc$ we can associate a unique normalized $\sbc$. Hereafter, we will refer to normalized $\sbc$ simply as $\sbc$.

The next result can be obtained by a straightforward modification of the proof 
of the variational characterization of $\cc$ and provides a variational characterization of $\sbc$ as critical points of the restriction of $U$ to the collision free configuration sphere.

\begin{lem}\label{thm:s-balanced-critical-points}
A (normalized) configuration vector $q$ is an $\sbc$  if and only if it is a critical point of 
$\widehat U\= U|_{\widehat{\mathbb S}}.$ \qed
\end{lem}

In the next result, we provide a representation of the Hessian quadratic form at any $\sbc$ w.r.t. to the the Euclidean as well as w.r.t. the mass scalar product. 
\begin{lem}\label{thm:Hessian-SBC}
	The Hessian of $\widehat  U:  \widehat{\mathbb S} \to \R$ at a critical point $q$ is the quadratic form on $T_q \widehat{\mathbb S}$ 
	that can be represented  w.r.t. the Euclidean  and  mass scalar product respectively by  the $dn$- dimensional matrices 
	\begin{align*}
	\widehat H(q) &= D^2 U(q) + U(q)\widehat SM,\\   H(q) &=M^{-1}D^2 U(q)+ \widehat SU(q) .
	\end{align*}
	\end{lem}
\begin{rem}
The latter representation is the natural choice if we are looking at  the Hessian at an $\sbc$  as the linearization of the gradient flow $D \widetilde \nabla \widehat U(q)$, where  $\widetilde \nabla$ is  the gradient  on $ \mathbb S $ with respect to the Riemannian metric  induced by the weighted mass-scalar product of the ambient manifold.  \qed
\end{rem}
By setting 
\[
r_{ij}\=|q_i-q_j|,\quad u_{ij}\= \dfrac{q_i-q_j}{|q_i-q_j|},
\]
a direct computation shows that the $(i,j)$-entry of the block symmetric matrix $D^2 U(q)$  is given by 
\begin{equation}\label{eq:block-structure-2}
D_{ij}\=
\dfrac{m_im_j}{r_{ij}^3}\left[\Id_d - 3 u_{ij}\trasp{u_{ij}}\right]\quad   \textrm{ for }\quad  i \neq j,  \quad \quad 
D_{ii}\=-\sum_{i \neq j} D_{ij}.
\end{equation} 
The group $\SO(d)$ acts diagonally on $\widehat{\mathbb S}$.  
Since the potential is rotationally invariant, we get that 
\[
D^2U(R q)= \trasp{R} D^2 U(q) R, \quad  R\in \SO(d).
\] 
In particular, the Hessian of a $\cc$ is $\SO(d)$ invariant, whereas for general $\sbc$ the Hessian is invariant by a proper subgroup of $\SO(d)$, which is the trivial subgroup in case $s_i\neq s_j$ for all $i\neq j$. 
%%%%%%%%%%%%%%%%%%%%%%%%%%%%%%%%%
%%%
%%%
%%%
%%%
%%%%%%%%%%%%%%%%%%%%%%%%%%%%%%%%%

\subsection{CSBC and Assumption (H1)}

In this subsection we provide a useful representation of the Hessian at a collinear $\sbc$ ($\csbc$ in shorthand notation) and we discuss Assumption~{\bf(H1)} that will be used throughout the paper.  

As a preliminary observation, we see from an easy inspection of Equation~\eqref{eq:s-balanced-conf} that the line containing each body $q_k$ of a $\csbc$  must lie in an eigenspace of the matrix $S$. % In fact, in the case of a $\csbc$ along a line  $L$, the force exerted is directed along the same line $L$; so, in order the line to be preserved by the vector-field defined in Equation~\eqref{eq:s-balanced-conf}, $L$ should be an eigenspace of $S$.
In particular, if the $s_j$'s are pairwise distinct, then the only lines containing $\csbc$ are the coordinate axis.  
 \begin{note}
 We can assume without loss of generality that each particle of the $\csbc$ $q$ is in the ``$j$-th eigenspace''  $\langle e_j\rangle$ of $S$, where as usual $e_j$ denotes the $j$-th vector of the canonical basis. Up to relabeling the $e_j$'s if necessary, we can assume that each body $q_k$ of the $\csbc$ $q$ lies on $\R\times (0)^{d-1}\subset \R^d$, and so the unit vectors $u_{ij}$ are all multiples of $e_1$. \qed
\end{note}
By a straightforward computation we get that each block $D_{ij}$ has the following representation
\[
D_{ij}= \dfrac{m_im_j}{|r_{ij}|^3}\begin{bmatrix}
-2 & 0 & \ldots & 0\\
0& 1 & \ldots & 0\\
\vdots & \vdots & \vdots & 0\\
0 & 0 & \ldots & 1	
\end{bmatrix}, \qquad D_{ii}= -\sum_{j \neq i } D_{ij}.
\]
After rearranging the  $dn$ coordinates of the configuration vector $q$ by collecting the components into groups of $n$ variables with all of the $e_1$ components first, the $e_2$ components second and so on, the Hessian at the   $\csbc$  is of the form
\[
H(q) =  \begin{bmatrix} -2 M^{-1} B(q) & & &  \\ &  M^{-1} B(q) & & \\ & & \ddots & \\ & & &  M^{-1} B(q)	
  \end{bmatrix}+
  \begin{bmatrix} s_j U(q) \Id_n & & &  \\ & s_1U(q)\Id_n & & \\ & & \ddots & \\ & & & U(q)\Id_n 
  \end{bmatrix}
\]
where  $B(q)$ is the $n \times n$ matrix whose $(i,j)$-entry is given by
\[
b_{ii}(q)=-\sum_{\substack{j=1\\ i \neq j}}^n \dfrac{m_im_j}{r_{ij}^3},\quad \quad b_{ij}(q)= \dfrac{m_im_j}{r_{ij}^3}.
\]
\begin{rem} \label{index:cc} 
We recall that the dimension of the configuration sphere is $N-1=d(n-1)-1= \dim T_q\mathbb S$ and so, for $d=2$ (planar case) we get $2n-3$, whereas for $d=3$ (spatial case), we get $3n-4$. In the particular case of $\cccc$ (corresponding  to $s_1=....=s_d=1$), the $(n \times n)$-block matrix $M^{-1} B(q)$ has Morse index $(n-1)$ (this result can be traced back to Conley) with largest eigenvalue $-U(q)$, and nullity $1$ with corresponding eigenvector $(1, \ldots, 1)$  (being transverse to the tangent space at $q$ to $\mathbb S$) which does not contribute to the inertia indices. Similarly, the block $-2M^{-1} B(q)$ has Morse coindex $(n-1)$ and nullity $1$, which once again does not contribute to the inertia indices. In particular, the inertia indices of the $\cccc$ $q$ are
\[
\big(\iMor{q}, \nullity{q}, \coiMor{q} \big)=\big((d-1)(n-2), d-1, n-2\big).
\]
Notice that $(d-1)(n-2)+(d-1) +(n-2)= d(n-1)-1= \dim T_q\mathbb S$. \qed
\end{rem}

Starting from Remark~\ref{index:cc}, in the next subsection we provide some interesting information about the inertia indices of $\csbc$ which will be then the starting point for the lower bounds for 
the number of $\sbc$ that we will prove in Section~\ref{sec:lower-bounds}. Hereafter, until the end of Section~\ref{sec:lower-bounds} we will always implicitly assume that 
\[
\text{{\bf(H1)\/}}  \qquad \qquad  \hspace{3cm} s_1> s_2> \ldots > s_{d-1}>s_d=1.\qquad \qquad \qquad \qquad \hspace{3cm}
\]
Such an assumption might appear at a first glance rather restrictive. However, it has several noteworthy advantages that we will exploit throughout the paper and that we now thoroughly discuss. 
Before commenting further on Assumption {\bf (H1)}, we shall notice that there is no loss of generality in assuming that the $s_i$'s be strictly decreasing and that $s_d =1$. This follows from the fact 
that we can always rearrange the vectors of the standard basis of $\R^d$ and that the problem is invariant under scaling. 

\begin{itemize}
\item Assumption {\bf(H1)} ``kills'' all symmetries of 
Equation~\eqref{eq:s-balanced-conf}, thus allowing to apply 
standard Morse theory arguments (rather than more complicated theories, such as e.g. equivariant Morse theory or Morse-Bott theory) when looking for lower bounds on the number of $\sbc$. 

\item The lower bounds on the number of $\sbc$ that one obtains assuming {\bf(H1)} still allow to give (possibly not optimal but still remarkable) lower bounds on the number of $\sbc$ in the general case via a 
suitable ``reduction to {\bf{(H1)}} argument'' as we now explain. Thus, suppose that for $d\ge 2$ positive real numbers $s_1>...>s_d=1$ and positive integers $\mu_1,\ldots,\mu_d\in \N$ are given. On $\R^D$, $D:= \mu_1+\ldots +\mu_d\ge d$, we consider the $S$-balanced configurations problem with 
\[
\widetilde S:= \diag \, (\underbrace{s_1,\ldots,s_1}_{\mu_1\text{-times}},\ldots,\underbrace{s_d,\ldots,s_d}_{\mu_d\text{-times}}).
\]
Notice that, under the assumptions above, Equation~\eqref{eq:s-balanced-conf} is invariant by the diagonal action on the configuration space of the subgroup of $\SO(D)$ given by the product group $\SO(\mu_1)\times ... \times \SO(\mu_d)$. For each $i=1,...,d$ we fix a line $L_i$ on the $\mu_i$-dimensional subspace of $\R^D$ relative to $s_i$, and restrict our attention to $\sbc$ which are contained in 
\[
Y:= \bigoplus_{i=1}^d L_i \cong \R^d.
\]
Ignoring all vanishing components of configuration vectors which lie in $Y$, we obtain that $\sbc$ which are contained in $Y$ satisfy Equation~\eqref{eq:s-balanced-conf} on $\R^d$ with 
$$S:= \text{diag} \, (s_1,...,s_d),$$
which is nothing else but {\bf{(H1)}} on a proper subspace of $\R^D$. We call this procedure the {\sc reduction to} {\bf{(H1)}} {\sc argument}. 

\item Applying equivariant theories for the $S$-balanced configurations problem in $\R^D$ seems hopeless in full generality due - among other facts - to the amount of cases that one should treat (corresponding to different choices of the $\mu_i$'s)  and to the fact that, for $d\ge 3$, the  diagonal action of the symmetry group $\SO(\mu_1)\times ... \times \SO(\mu_d)$ of Equation~\eqref{eq:s-balanced-conf} on the configuration space is not free, thus yielding a quotient which is not a manifold (actually, not even an orbifold). 
%Even worse, in many cases the quotient, is not even an orbifold but usually and Alexandrov space requiring a delicate stratified Morse theory to be defined. 

\item Theories like Morse-Bott theory, which allow to deal with functions having ``non-degenerate'' critical manifolds rather than critical points, do not take into account the symmetries 
of the problem and thus necessarily lead to weaker results. For instance, it is possible to show that for the $S$-balanced configurations problem on $\R^4$ with $S=\text{diag}(s_1,s_1,1,1)$,
the lower bounds on the number of $\sbc$ that one obtains via the reduction to {\bf{(H1)}} argument are precisely twice as much as the lower bounds that one obtains by Morse-Bott theory. 

\item All methods involving equivariant or stratified Morse theories are essentially based on equivariant homology and intersection homology theories, which are 
not computationally feasible in all cases. 

\item If the singular set consists only of collinear configurations, under suitable assumptions,  one could try to overcome these difficulties by quotienting out the configuration space by the group action after removing the singular set. We will say more about this in Section \ref{sec:homology}, where we  will compare the lower bounds on the number of $\sbc$ in $\R^4$ with $S=\text{diag}\, (s_1,s_1,1,1)$ obtained implementing such a strategy with the ones obtained via the reduction to {\bf{(H1)}} argument (thus, by considering the $S$-balanced configurations problem in $\R^2$ with matrix $\diag(s_1,1)$).

\item Even if one is able to give lower bounds on the number of $\sbc$ not assuming {\bf{(H1)}}, for instance via equivariant Morse theory or Morse-Bott theory,  one has additionally to show that the $\sbc$ that one finds are actually not $\cc$. Notice indeed that,
for $\sbc$ which are contained in a subspace of $\R^D$ corresponding to some $s_i$, Equation~\eqref{eq:s-balanced-conf} reduces to the central configurations equation. This is the same problem that one has to face when trying to exclude that spatial $\cc$ are actually not planar: despite the lower bounds provided by Pacella \cite{Pac86}, it is so far not known whether there is more than one spatial  $\cc$ which is not planar. 
%In a similar vein, all $\sbc$ that one finds using equivariant theory or Morse-Bott theory might actually be $\cc$. 

Assuming {\bf{(H1)}}, instead, we easily overcome this problem: indeed, under Assumption {\bf{(H1)}}, every solution of \eqref{eq:s-balanced-conf} which is not collinear cannot be a $\cc$. 

\item Under Assumption {\bf{(H1)}}, solutions of Equation~\eqref{eq:s-balanced-conf} yield relative equilibria for the $n$-body problem in $\R^{2d}$ via ``complexification'', which can be thought of as the inverse  procedure to the reduction to {\bf{(H1)}} argument. Indeed, $\sbc$ 
on $\R^d$ for $S=\text{diag}\, (s_1,s_2,...,s_{d-1},1)$ are in particular $\sbc$ in $\R^{2d}\cong \C^d$ with matrix $\text{diag}\, (s_1,s_1,s_2,s_2,...,s_{d-1},s_{d-1},1,1)$.
\end{itemize}

%%%%%%%%%%%%

\subsection{CSBC and corresponding CCC: spectra and inertia indices}
\label{subsec:inertiaindices}

In what follows, we refer to a $\csbc$ in which every particle  $q_k$ belongs to $\langle e_j\rangle $ (the eigenspace corresponding to the eigenvalue  $s_j$) just as $s_j-\csbc$. 	
In this subsection, after describing the relation between $s_j-\csbc$ and the corresponding $\cccc$, we provide an explicit description of their inertia indices. 

\begin{lem}\label{thm:cc-sbc-collinear}
Let $q$ be a $s$-$\csbc$, where  $s=s_j$ for some $j=1,...,d$. Then: 
\begin{itemize}
	\item $\widehat q\= \sqrt{s}\,q$ is a $\cccc$.
	\item $M^{-1} B(q)+ s U(q)I_n= s\sqrt{s}\left[M^{-1}B(\widehat q)+ U(\widehat q)I_n\right]$.
\end{itemize}
\end{lem}
\begin{proof}
%In the (1D) collinear case the symmetric matrix $S $ is just the scalar $s$; so,  we get that  $S= s \Id_n$.
 By assumption, the $\csbc$ $q$ lies in the eigenspace of $S$ corresponding to the eigenvalue $s$ and  $\mnorm{\sqrt{s}\, q}=1$. Thus, by a direct computation, we get 
\[
\nabla U(\sqrt{s}\cdot q ) = s^{-1} \nabla U(q) = - s^{-1} s U(q) M  q  = -U(q) M q = -U(\sqrt{s}\cdot q) M   \sqrt{s} \cdot q .
\]
Setting $\widehat q\= \sqrt{s}\cdot q$, we obtain $\nabla U(\widehat q)=-U(\widehat q)M \widehat q$ which implies that $\widehat q$ is a $\cccc$. 
%As already observed, if $q$ is a   collinear $\sbc$ then $\widehat q\= \sqrt{s}\cdot q$ is a   collinear central configuration.
The homogeneity of $U$ implies that 
\[
B(q) = s \sqrt{s}\cdot B(\widehat q)
\]
and hence the eigenvalues of $B(q)$ are easily obtained by
multiplying each eigenvalue of $B(\widehat q)$ by $s\sqrt{s}$.
\end{proof}
\begin{rem}
As we shall see below, from Lemma~\ref{thm:cc-sbc-collinear} it follows that the nullity of $q$ is generically zero, but might be non-zero for certain choices of the $s_i$'s. Moreover, the coindex of a $\csbc$ can possibly be greater than the coindex of the corresponding $\cccc$, as it could increase (but not decrease) for varying $s$.  Analogously, the  index of a $\csbc$ can possibly be lower than the index of the corresponding $\cccc$, as it could decrease (but not increase) for varying $s$. Loosely speaking, since the eigenvalues of $S$ are assumed to be greater than $1$,  we get that adding to the block matrix $M^{-1}B(q)$ a block diagonal matrix of the form $s_i U(q)\Id_n$, we can move part of the spectrum located on the negative real line to the positive thus changing the inertia indices. 
 \end{rem}
Another important property which is pointed out by Lemma~\ref{thm:morse-index-co-balanced} is the different role (with respect to the inertia indices of a $\csbc$) played by the different eigenspaces. In fact, for an $s-\csbc$  the  index is maximal, and the coindex is minimal, for $s=s_1$.
\begin{lem}\label{thm:morse-index-co-balanced}
Let $q$ be a $\csbc$. The following facts hold.
\begin{enumerate}
	\item The index of $q$ is at most $(d-1)(n-1)$ and the  coindex of $q$ is at least $(n-2)$
\item If $q$ is an $s_1-\csbc$, then the  Morse index is precisely $(d-1)(n-1)$ and the Morse coindex is precisely $n-2$. Furthermore, if $q$ is an $s_j-\csbc$ and  $j<d$, then $q$ cannot
be a local minimum of $\widehat U$.
\end{enumerate}
\end{lem}
\begin{proof}
We assume that $q$ is a $s_j$-$\csbc$ for some $j\in \{1,\ldots , d\}$. By using Lemma~\ref{thm:cc-sbc-collinear},  for  $i \neq j$  we get that
\begin{align*}
M ^{-1}B(q) + s_i U(q)I_n &= s_j \sqrt{s_j}\left[M ^{-1}B (\widehat q)+ \left(\dfrac{s_i}{s_j}\right) U(\widehat q) I_n\right]\\ 
&=s_j \sqrt{s_j}\left[M ^{-1}B (\widehat q)+ U(\widehat q)I_n+ U(\widehat q)\left(\dfrac{s_i-s_j}{s_j}\right)I_n\right]
%\quad \textrm{ and }\\
%-2M ^{-1}B(q) + s_i U(q)= s_j \sqrt{s_j}\left[-2M ^{-1}B (\widehat q)+ U(\widehat q)+ U(\widehat q)\left(\dfrac{s_i-s_j}{s_j}\right)\right]
%\dfrac{s_i}{s_j} U(\widehat q)\right] + (s_i-s_1)U(q)\\
%=s_1\sqrt{s_1}\left[M ^{-1}B(\widehat q)+ U(\widehat q)\right] + (s_i-s_1) U(\widehat q)
\end{align*}
where $\widehat q$ is the $\cccc$ corresponding to $q$.
Depending on $i$ and $j$, the quantity 
\[
\left(\dfrac{s_i-s_j}{s_j}\right)
\] 
can be positive or negative. For $j=1$ (corresponding to one of the two extremal cases), this term is always negative (i.e. independent of $i$). Thus, also the nullity of the central configuration $\widehat q$ contributes to the Morse index of the $s_1$-$\csbc$ (in fact, the contribution is $1$ for each of the $(d-1)$ blocks). The contribution to the coindex coming from the first block of the Hessian is $n-2$. Arguing analogously, we see that for $j>1$ the Morse index is strictly less than $(d-1)(n-1)$, whereas the coindex is strictly greater than $n-2$. This follows from the fact that, for $i<j$, the above quantity is strictly positive and hence the nullity of the $\cc$ $\widehat q$ contributes for such $i$'s to the Morse coindex. Finally, for $j<d$, the above quantity for $i=d$ is strictly negative, thus the nullity of $\widehat q$ 
contributes to the Morse index. 
\end{proof}
The next result describes the properties of a $\csbc$ in terms of the spectra of the corresponding $\cccc$ (or $s_d-\csbc$ which is the same) as well as of the eigenvalues of $S$. 
\begin{lem}\label{thm:balanced-central-index-relation}
Let $\widehat q$ be an $s_d-\csbc$, and let
\[
\eta\= \textrm{ smallest eigenvalue of } M ^{-1}B(\widehat q).
\]
Then:
\begin{itemize}
	\item $\widehat q$ is a local minimum of $\widehat U$ provided that $\eta + s_{d-1} U(\widehat q)>0$
	\item $\widehat q$ has non-vanishing Morse index if and only if $\eta + s_{d-1} U(\widehat   q)<0 $. 	
	\item Suppose $\eta + s_1 U(\widehat q) <0$. For every $j=1,...,d$, let $\widehat q_j$ be the $s_j$-$\csbc$ whose ordering of the masses coincides with the one of $\widehat q$. Then, for any coordinate subspace\footnote{This means $E=\text{span}\, \{e_{i_1},...,e_{i_h}\}$ for some $h\ge 1$ and $i_1,...,i_h$ pairwise distinct.}  
	$E\subset \R^d$ such that $\langle e_j \rangle \subset E$, $\dim E\ge 2$, we have that  $\widehat q_j$ 
	is not a local minimum of $\widehat U|_{\widehat{\mathbb S} \cap E^n}.$
\end{itemize}
\end{lem}
%\begin{rmk}
%It is worth noticing that, if $s_d \neq 1$, then the first spectral condition should be replaced by the analogous  condition $s_d\eta + s_{d-1} U(\widehat  q)>0 $ whilst the second by $s_d\eta + s_1 U(\widehat  q)<0 $.   Furthermore, we notice also  that under the  second condition part of the spectrum of  $ M ^{-1} B(\widehat q)$ is negative, thus giving a  non-trivial contribution to the Morse index. It is interesting to observe that it can be shown that $\eta +s_1U(\widehat q)<0$ iff $q$ has non-vanishing Morse index. 
%\end{rmk}

\begin{proof}
For $i<d$, the smallest eigenvalue of the $i$-th block of the Hessian
\begin{align*}
M ^{-1} B(\widehat q)+ s_i U(\widehat q) I_n
\end{align*}
is given by 
$$\eta + s_i U(\widehat q).$$ 
The first two claims readily follows. As far as the last claim is concerned, we observe that for the $s_j-\csbc$ $\widehat q_j$ whose ordering of the masses coincides with the one of $\widehat q$ and for every $i \neq j$ we have that the corresponding block of the Hessian of $U$ at $\widehat q_j$ is given by
\[
M ^{-1} B(\widehat q_j) + s_i U(\widehat q_j)I_n  =  s_j \sqrt s_j \left[M ^{-1}	B(\widehat q)+ \left(\dfrac{s_i}{s_j}\right) U(\widehat q)I_n\right].
\]
In particular, the smallest eigenvalue of $M ^{-1} B(\widehat q_j) + s_i U(\widehat q_j)I_n$ is (up to the positive factor $s_j\sqrt{s_j}$) given by
\[
\eta + \dfrac{s_i}{s_j} U(\widehat q).
\]
We now estimate 
\[\eta + \dfrac{s_i}{s_j} U(\widehat q) <\eta + \dfrac{s_1}{s_j} U(\widehat q)<\eta + \dfrac{s_1}{s_d}U(\widehat q)=\eta + s_1U(\widehat q)<0, 
\]
thus showing that the considered block has at least one negative eigenvalue. The claim readily follows.
\end{proof}

Clearly, the value of $U(\widehat q)$, as well as the matrix $B(\widehat q)$ and its eigenvalues, depend on the choice of the $\cccc$ (i.e. on the choice of the ordering of the masses), unless all masses are equal. Therefore, the conditions above are a priori different 
for different ordering of the masses, and there is in general no direct way of comparing the inertia indices of $\csbc$ with different ordering of the masses. Nevertheless, we can still use this information 
to deduce existence of non-collinear $\sbc$ as we now show. 

\begin{cor}
\label{cor:continua}
Under the notation above, if all masses are equal and
$$1< s_1 < - \frac{\eta}{U(\widehat q)},$$ then there are at least $\frac{d(d-1)}{2}$ planar non-collinear $\sbc$. In particular, the $n$-body problem in $\R^{2d}\cong \C^d$ has at least $\frac{d(d-1)}{2}$ continua of relative equilibria in which periodic motions form a dense subset. If the masses are not all equal, then the same conclusion holds provided one replaces the spectral condition above with 
$$1<s_1 < \min \Big \{  - \frac{\eta}{U(\widehat q)} \ \Big | \ \widehat q \ 	\ \cccc\Big \}.$$
\end{cor}
\begin{proof}
For $1\le i <j\le d$ fixed, we consider the $(e_i,e_j)$-plane in $\R^d$. If all masses are equal, then Lemma~\ref{thm:balanced-central-index-relation} implies that the global minimum of $\widehat U$ restricted to $\text{span}\, \{e_i,e_j\}$
cannot be attained at a $\csbc$, and hence must be attained at a planar non-collinear $\sbc$. The first part of the claim follows, since 
the cardinality of the set $\{(i,j) \ |\ 1\le i<j\le d\}$
is precisely $\frac{d(d-1)}{2}$. We now set 
$$\mathcal S := \left \{(s_1,...,s_{d-1}) \ \Big |\ -\dfrac{\eta}{U(\widehat q)}>s_1> s_2> ... >s_{d-1}\right \}.$$ 
The stability property of minima implies that the maps
$$\mathcal S \ni (s_1,...,s_{d-1}) \mapsto \min \widehat U|_{\text{span}\, \{e_i,e_j\}},\quad 1\le i<j\le d,$$
yield the desired continua of relative equilibria. Such relative equilibria will be periodic motions whenever $s_1,...,s_{d-1},1$ are rationally dependent, and quasi-periodic motions otherwise.
This completes the proof in the case all masses are equal. The proof is analogous in the case of masses which are not all equal. Indeed, the new spectral gap condition implies that none of the $\csbc$ 
can be a local minimum for the restriction of $\widehat U$ to $\text{span}\, \{e_i,e_j\}$ as well. The details are left to the reader. 
\end{proof}

\begin{rmk}
Corollary~\ref{cor:continua} holds for any choices of the masses $m_1,...,m_n$, that is, independently of the fact that $\sbc$ are degenerate or not. For $n=3$ and $m_1=m_2\neq m_3$,  one retrieves the quasi-periodic motions corresponding to isosceles triangles in a prescribed $(e_i,e_j)$-plane. (Cfr. \cite[Pag.125]{Moe14}, for further details).

Actually, by using similar arguments, it is possible to prove the existence of others continua of relative equilibria coming from minmax points (that must exist for topological reasons).  
  \qed
\end{rmk}

\begin{rem}
To any planar non-collinear $\cc$ we can associate two integers, the so-called planar index and normal index. As shown by Moeckel in \cite[Pag. 69]{Moe94} (by  using an argument originally due to Pacella) estimating  from below the trace of $-M^{-1} B(q)$ at such a planar non-collinear $\cc$, it is possible to provide a lower bound on the normal index of a central configuration. A direct consequence of this estimate is that the global minimum of $\widehat U$ doesn't occur at a planar $\cc$, and so proving  the existence of at least one non-planar $\cc$.  It would be interesting to understand if an analogous argument can be carried over to the case of planar non-collinear $\sbc$.\qed
\end{rem}
We conclude this section by determining  the inertia indices of  $\csbc$ in case $d=2$. Thus, let $S=\diag\, (s_1,1)$ with $s_1>1$. By Lemma~\ref{thm:morse-index-co-balanced} we get that the inertia indices of a $s_1-\csbc$, denoted  $q$, are 
\[
\nullity{q} = 0,\quad \coiMor{q} = n-2,\quad \iMor{q} =n-1,
\]
whereas for a $s_2-\csbc$ (which are nothing else but $\cccc$, thus denoted $\widehat q$) we have
\[
\nullity{\widehat q} \ge 0,\quad \coiMor{\widehat q} \ge n-2,\quad \iMor{\widehat q} \le n-1.
\]
In the next result we provide a complete characterization of the inertia indices of $s_2-\csbc$ depending on $s_1$ and on the spectrum of $M^{-1}B$ at $\widehat q$. Thus, let $\widehat q$ be a $s_2-\csbc$ (or, equivalently, a $\cccc$), and  denote by
\[
\eta_k(\widehat q) < ... < \eta_1(\widehat q) < -U( \widehat q) < 0
\]
the eigenvalues of the matrix $M ^{-1}B(\widehat q)$. Even though the $\eta_i$'s depend on the choice of $\widehat q$, to ease the notation we hereafter drop the dependence on $\widehat q$. 
It is well known that the eigenvalues $-U(\widehat q)$ and $0$ are simple, but any other eigenvalue might have multiplicity. For $j=1,\ldots,k$ we denote by $\alpha_j$ the multiplicity of $\eta_j$ as an eigenvalue of $M^{-1}B(\widehat q)$.  We also set $\eta_0:= -U(\widehat q)$ and $\alpha_0:=1$. 
\begin{prop}\label{lem:indicescollinear}
Let $\widehat q$ be an $s_2-\csbc$. Then, the following hold.
\begin{enumerate}
\item  If $-\eta_j < s_1 U(\widehat q) < -\eta_{j+1}$ for some $j\in \{0,...,k-1\}$,  then 
\[
\nullity{\widehat q} = 0,\quad \coiMor{\widehat q} = n-2 + \sum_{i=0}^j \alpha_i,\quad \iMor{\widehat q} =n-1-\sum_{i=0}^j \alpha_i.
\]
\item If $s_1 U( \widehat q) = -\eta_j$ for some $j \in \{1,...,k\}$,  then 
\[\nullity{\widehat q} = \alpha_j,\quad\coiMor{\widehat q} = n-2 + \sum_{i=0}^{j-1} \alpha_i,\quad \iMor{\widehat q} =n-1-\sum_{i=0}^j \alpha_i.
\]
In particular, in this case, $\widehat q$ is a degenerate critical point of $\widehat U$. 
\item If $ s_1 U( \widehat q) +\eta_k>0$, then 
\[
\nullity{\widehat q} = 0,\quad\coiMor{\widehat q} = 2n-3,\quad\iMor{\widehat q} =0.
\]
In particular, in this case, $\widehat q$ is a local minimum of $\widehat U$.
\end{enumerate}
\end{prop}
\begin{proof}
	In order to prove  the first claim we start observing that by adding  the term $s_1 U(\widehat q)\Id_n$ to the matrix $M ^{-1} B(\widehat q)$ part of the spectrum becomes positive. Since the assumption is equivalent to $\eta_{j+1} < -s_1 U(\widehat q) < \eta_j$, we get
	\[
	\eta_{j+1} +s_1 U(\widehat q) <0<\eta_j +s_1U(\widehat q).
	\]
	So, to the Morse coindex of $q$ we add the contribution provided by the first $(j+1)$ eigenvalues of $M ^{-1} B(\widehat q)$ which become positive by adding $s_1 U(\widehat q)$. Clearly, the same contribution has to be subtracted from the Morse index. 
The proof of the second item is completely analogous. We just need to observe that the $(j+1)$-th eigenvalue $\eta_j$ of $M ^{-1} B(\widehat q)$ gives a non-trivial contribution to the nullity. The third item follows straightforwardly since, under the assumption $ s_1 U(\widehat q) +\eta_k>0$, the whole spectrum of $M ^{-1} B(\widehat q)$ moves into the positive real axis. 
	\end{proof}
\begin{rmk}
As a direct consequence of Proposition~\ref{lem:indicescollinear}, we get that there are only finitely many values of $s_1$ for which $s_2-\csbc$ can be degenerate as critical points of $\widehat U$.
In other words, $s_2-\csbc$ are generically non-degenerate.
We shall however remark that the values of $s_1$ that make an $s_2-\csbc$ degenerate depend on the $s_2-\csbc$, unless the masses are all equal. 
In a similar vein, where and how big the jumps on the Morse index and coindex are depend on the $s_2-\csbc$. 

A result analogous to  Proposition~\ref{lem:indicescollinear}, which will be omitted to keep the exposition as simple as possible, can be proved also for $d>2$, thus showing as a corollary that  $\csbc$ are generically non-degenerate. \qed
\end{rmk}

%%%%%%%%%%%%%%%%%%%%%%%%
%%
%%
%%
%%
%%%%%%%%%%%%%%%%%%%%%%%%

\section{Estimates on the coefficients of the Poincar\'e polynomial}\label{sec:poincare}

In this section we provide some elementary but rather not trivial (asymptotic) estimates on the coefficients of the Poincar\'e polynomial of the collision free configuration sphere. The interest in such estimates relies mainly on the following facts:
\begin{itemize}
\item They yield non-trivial lower bounds for the number of $\sbc$,
since they show that also $s_j-\csbc$, $j \ge 2,$ can contribute to the count of critical points.
\item They have highly non-trivial and rather unexpected qualitative consequences on the count of critical points. Indeed, in ``some cases'' $s_j-\csbc$, $j \ge 2$, will contribute more than $s_1$-collinear $\sbc$ (this should be compared with \cite[Pag. 151]{Moe14}).  
\item They might have applications in the study of limiting problems (i.e. with very large number of bodies).
\end{itemize}

We start recalling the Poincar\'e polynomial of the collision free configuration sphere $\widehat{\mathbb S}$. As this result is  well-known, its proof will be omitted.
\begin{lem}\label{thm_:poincare-polynomial}
The Poincar\'e polynomial of $ \widehat{\mathbb S}$ is given by 
\begin{equation}\label{eq:poincarepolynomial}
\hspace{40mm} \, P(t) = (1+ t^{d-1})(1+2t^{d-1})\cdot ... \cdot (1 + (n-1)t^{d-1}). \, \hspace{40mm} \qed
\end{equation}
\end{lem}
We are now ready to prove some  results about the growth rate for the coefficient $c^{(n)}_j$ with  $j=0,...,n-1$ of  the polynomial 
\[
p_n(z) \= (1+z)(1+2z) \cdot ... \cdot (1+(n-1)z),\quad n\in\N.
\]
It is worth observing  that the Poincar\'e polynomial $P(t)$ of $\widehat{\mathbb S}$ is nothing else but $p_n(t^{d-1})$. For notational convenience we hereafter set $c_j^{(n)}=0$ for every $j \ge n$. 

\begin{prop}\label{lem:coeffofpoincarepolynomial}
The following statements hold:
\begin{enumerate}
\item $\displaystyle \sum_{j=0}^{n-1} c_j^{(n)} = n!$ for every $n\in\N$
\item $c_j^{(n)}\le \dfrac{n!}{2}$ for every $j\in \N_0$ and every $n\ge 2$
\item For fixed $j\in \N_0$, we get 
\[
\lim_{n\to +\infty} \frac{c_j^{(n)}}{n!} =0.
\]
\item For every $n\in \N$, we get 
\[
c_{n-1}^{(n)}=(n-1)! \quad \text{and} \quad c_{n-2}^{(n)} \sim (\gamma + \log n)(n-1)!
\]
where 
\[
\gamma\= \lim_{n\to +\infty} \left (\sum_{j=1}^n \frac 1j - \log n\right )
\] 
is the {\em Euler-Mascheroni constant\/}. More generally, for every fixed $j\in \N$ we get
\[
c_{n-j}^{(n)} \lesssim (\gamma + \log n)^{j-1} (n-1)!
\]
where ``$\lesssim$'' means inequality up to some constant independent on $j$ and $n$. Hence, in particular 
\[
\lim_{n\to +\infty} \frac{c_{n-j}^{(n)}}{n!} =0.
\]
\end{enumerate}
\end{prop}

\begin{rmk}
Part 1 of Proposition~\ref{lem:coeffofpoincarepolynomial} yields the following interesting representation of the factorial of $n$:
\[
n! = 1 + \sum_{i_1=1}^{n-1} i_1 + \sum_{1\le i_1<i_2\le n-1}\hspace{-3mm} i_1i_2 \ \ + ... + \sum_{1\le i_1<...<i_{n-2}\le n-1}\hspace{-5mm} i_1\cdot ... \cdot i_{n-2} \ \ + (n-1)!
\]
To our best knowledge, such a representation has never appeared in the literature. Notice also that Part 1 immediately implies that 
all but at most one coefficients $c_j^{(n)}$ are smaller than $n!/2$. We observe that the inequality in Part 2 is strict as soon as $n \ge 4$.
\end{rmk}

\begin{proof}
$\ $
\\
\begin{enumerate}
\item The claim follows trivially by evaluating 
$$p_n(z)= (1+z)(1+2z)\cdot ... \cdot (1+(n-1)z)= \sum_{j=0}^{n-1} c_j^{(n)} z^j$$
at $z=1$. However we provide a proof by induction over $n\in \N$. For $n=1$ the claim is obvious. Now, suppose the claim holds true for $n$. Since
\[
\sum_{j=0}^n c_j^{(n+1)} z^j = (1+nz) \sum_{j=0}^{n-1} c_j^{(n)} z^j,
\]
evaluating at $z=1$ yields by inductive assumption
\[
\sum_{j=0}^n c_j^{(n+1)} =  \sum_{j=0}^{n-1} c_j^{(n)} z^j + n  \sum_{j=0}^{n-1} c_j^{(n)} z^j = n! + n n! = (n+1)n! = (n+1)!
\]
\item Once again by induction over $n$. For $n=2$ the claim is trivially satisfied. Assuming that the claim hold for $n$ yields for every $j\in \N_0$:
\[
c_j^{(n+1)} = nc_{j-1}^{(n)} + c_j^{(n)} \le  n \dfrac{n!}{2} + \dfrac{n!}{2} = \dfrac{(n+1)!}{2}.
\]
\item By induction over $j\in \N_0$. The claim is obvious for $j=0$ as $c^{(n)}_0=1$ for all $n\in\N_0$. Now, suppose that the claim hold for $j$. 
Then, for $\epsilon>0$ there exists  $n_0=n_0(\epsilon)\in \N$ such that $c_j^{(n)}/n!<\epsilon$ for every $n\ge n_0$, and hence using 2. we get  for every 
$n \ge \max \{n_0, \frac{1}{2\epsilon} - 2\}$:
\begin{align*}
\dfrac{c_{j+1}^{(n+1)}}{(n+1)!}& = n \dfrac{c_j^{(n)}}{(n+1)!} + \frac{c_{j+1}^{(n)}}{(n+1)!} \\ 
					 &= \dfrac{n}{n+1} \dfrac{c_j^{(n)}}{n!} + \dfrac{c_{j+1}^{(n)}}{(n+1)!}\\
					 &< \dfrac{n}{n+1} \epsilon + \dfrac{1}{2(n+1)}\\
					&<2\epsilon.
\end{align*}
\begin{rem}
Alternatively, one can show that for $j\in \N$ fixed, the coefficient $c_j^{(n)}$ has a polynomial growth in $n$.
\end{rem}
\item The first claim is trivial. As far as $c_{n-2}^{(n)}$ is concerned, we compute 
\begin{equation}
c_{n-2}^{(n)}  = \sum_{i=1}^{n-1} 1\cdot 2\cdot ...\cdot  \widehat i \cdot ... \cdot (n-1)
		= (n-1)! \cdot \sum_{i=1}^{n-1} \dfrac 1i
		\sim (n-1)! \cdot (\gamma + \log (n-1))
\end{equation}
and the claim follows. Now, for fixed $j\ge 2$ we compute
\begin{align*}
c^{(n)}_{n-j} &= \sum_{1\le i_1<...<i_{j-1}\le n-1} 1\cdot ...\cdot \widehat i_1 \cdot ... \cdot \widehat i_{j-1} \cdot ... \cdot (n-1)\\
		&= (n-1)! \sum_{1\le i_1<...<i_{j-1}\le n-1} \frac{1}{i_1 \cdot ... \cdot i_{j-1}}\\
		&= (n-1)! \sum_{i_1=1}^{n-1-j} \frac{1}{i_1} \sum_{i_2=i_1+1}^{n-j} \frac{1}{i_2} \cdot  ... \cdot \sum_{i_{j-1}=i_{j-2}+1}^{n-1} \frac{1}{i_{j-1}}\\
		&\lesssim (n-1)! \sum_{i_1=1}^{n-1-j} \frac{1}{i_1} \sum_{i_2=i_1+1}^{n-j} \frac{1}{i_2} \cdot  ... \cdot \sum_{i_{j-2}=i_{j-3}+1}^{n-2} \frac{1}{i_{j-2}} \cdot (\gamma + \log(n-1))\\
		&\lesssim ... \\
		&\lesssim (n-1)! \prod_{k=1}^{j-1} (\gamma + \log (n- k)) \\
		&<(n-1)! (\gamma + \log (n- 1))^{j-1}.
\end{align*}
Finally, for fixed $j\in \N$
\begin{equation*}
\frac{c^n_{n-j}}{n!} \lesssim \frac{(n-1)! (\gamma + \log (n- 1))^{j-1}}{n!} = \frac{(\gamma + \log (n- 1))^{j-1}}{n} \to 0, \quad \text{as} \ n\to +\infty.
\qedhere
\end{equation*}
\end{enumerate}
\end{proof}
\begin{rem}
For $j\in \N$ sufficiently large the estimate provided by the last item in  Proposition~\ref{lem:coeffofpoincarepolynomial}, is very far from being optimal. In fact, as we shall see below, a much stronger statement holds. More precisely,  for every sequence $\{j_n\}_n\subset \N$ we have 
\begin{equation}
\lim_{n\to +\infty} \dfrac{c^{(n)}_{n-j_n-1}}{n!} =0.
\label{eq:asymptoticgrowth}
\end{equation}
The last two items of Proposition~\ref{lem:coeffofpoincarepolynomial} confirm the validity of \eqref{eq:asymptoticgrowth} when $\{n-j_n\}_n$ and $\{j_n\}_n$ are uniformly bounded. In the case of positively divergent sequences $j_n$, the behavior is described by Proposition~\ref{prop:asymptoticgrowth}. \qed
\end{rem}

\begin{prop}\label{prop:asymptoticgrowth}
Let $\{j_n\}_n\subset \N$ be a positively divergent sequence.   Then, the following hold.
\begin{enumerate}
\item For every $\epsilon >0$, 
$$\lim_{n\to +\infty} \frac{c^{(n)}_{n-j_n-1}}{n^\epsilon (n-1)!} =0.$$
In particular the asymptotic condition~\eqref{eq:asymptoticgrowth} holds. 
\item If $j_n \gtrsim \log \log n$, then 
$$\lim_{n\to +\infty} \frac{c^{(n)}_{n-j_n-1}}{(n-1)!} =0.$$
\item If $j_n\gtrsim \log n$, then for every $k\in \N$ fixed 
$$\lim_{n\to +\infty} \frac{c^{(n)}_{n-j_n-1}}{(n-k)!} =0.$$
\end{enumerate} 
\end{prop}

\begin{rmk}
Proposition~\ref{prop:asymptoticgrowth} states that every sequence $\{c^{(n)}_{k_n}\}_n$ of coefficients of the polynomials $p_n(z)$ grows slower than $n!$. Moreover, if the sequence $\{k_n\}_n$ does not grow ``too fast'', then we get even a slower growth; in fact,  $c^{(n)}_{k_n}$ grows slower than $(n-1)!$ if $k_n \lesssim n-\log \log n$ and slower than $(n-k)!$ for every $k\in \N$ provided $k_n\lesssim n-\log n$. Notice that for $\{k_n\}_n$ bounded we have polynomial growth, and that $(n-1)!$ is precisely the coefficient of $p_n(z)$ of degree $(n-1)$. 
\qed 
\end{rmk}
Before proving the proposition we  observe that for every $n\in \N$ and every $j\le n-1$: 
\begin{align}
c^{(n)}_{n-j-1} &= \sum_{1\le i_1<...<i_{j}\le n-1} 1\cdot ... \cdot \widehat i_1 \cdot ... \cdot \widehat i_{j}\cdot ... \cdot (n-1)\nonumber \\
			&= (n-1)! \sum_{1\le i_1<...<i_{j}\le n-1} \frac{1}{i_1\cdot ...\cdot i_{j}}\nonumber \\
			&= (n-1)! \sum_{i_1=1}^{n-j-1} \frac{1}{i_1} \cdot \sum_{i_2=i_1+1}^{n-j} \frac{1}{i_2}\cdot ... \cdot \sum_{i_{j}=i_{j-1}+1}^{n-1} \frac{1}{i_{j}}\nonumber  \\
			&\lesssim (n-1)! \int_1^n \dfrac{1}{i_1}\cdot \int_{i_1}^n \frac{1}{i_2}\cdot ... \cdot \int_{i_{j-1}}^n \frac{1}{i_{j}}\ \  \diff i_{j}\cdot ... \cdot \diff i_1. \label{firstestimate}
\end{align}
The key step is the following result, whose proof is postponed to  Appendix~\ref{sec:appendix}. 
\begin{lem}\label{thm:integrale-iterato}
For every $n,j\in \N$ we have 
$$  \int_1^n \frac{1}{i_1}\cdot \int_{i_1}^n \frac{1}{i_2}\cdot ... \cdot \int_{i_{j-1}}^n \frac{1}{i_{j}}\ \  \diff i_{j}\cdot ... \cdot \diff i_1 = \frac{1}{j!}\log^{j}(n).  $$
\end{lem}
%\begin{proof}
%The proof can be divided into the following two steps.\\
%\textbf{Step 1.} Prove that the following equalities hold
%\begin{align}
% &\int_1^n \frac{1}{i_1}\cdot \int_{i_1}^n \frac{1}{i_2}\cdot ... \cdot \int_{i_{j-1}}^n \frac{1}{i_{j}}\ \  \diff i_{j}\cdot ... \cdot \diff i_1 = \alpha_j \cdot \log^{j}(n), \label{1}\\
% &\int_{i_0}^n \frac{1}{i_1} \cdot \int_{i_1}^n \frac{1}{i_2}\cdot  ... \cdot \int_{i_{j-1}}^n \frac{1}{i_{j}} \ \diff i_{j}\cdot ... \cdot \diff i_1 = \sum_{k=0}^j \frac{(-1)^k}{k!} \alpha_{j-k} \log^{j-k}(n) \log^k (i_0), \label{2}
%\end{align}
%where the $\alpha_j$'s satisfy the recursive relation 
%\begin{equation}
%\left \{ \begin{array}{l} \alpha_0= \alpha_1 =1,\\ \alpha_{j+1}= \displaystyle \sum_{k=0}^j \frac{(-1)^k}{(k+1)!} \alpha_{j-k}, \quad \forall j\ge 1.\end{array}\right .
%\label{recursiverelation}
%\end{equation}
%\textbf{Step 2.} Prove that $\alpha_j=\frac{1}{j!}$ for every $j\in\N$.
%
% For further details on the proof we refer the interested reader to Appendix~\ref{app:integrale}.
%\end{proof}
\begin{proof}[Proof of Proposition~\ref{prop:asymptoticgrowth}]
$\ $\\
\begin{enumerate}
\item Using Equation~\eqref{firstestimate}, Lemma \ref{thm:integrale-iterato}, and Stirling's approximation of $n!$, we get 
\begin{align*}
\frac{c^{(n)}_{n-j_n-1}}{n^\epsilon (n-1)!} &\lesssim \frac{\log^{j_n}(n)}{n^\epsilon j_n!} \\
 &\sim \frac{e^{\log (j_n \cdot \log n)}}{n^\epsilon \sqrt{2\pi j_n} \big ( \frac{j_n}{e}\big )^{j_n}}\\
			&= \frac{e^{\log (j_n \cdot \log n)}}{e^{\epsilon \log n + \frac 12 \log (2\pi j_n) + j_n \log j_n - j_n}}\\
			&= e^{\log j_n + \log \log n- \epsilon \log n - \frac 12 \log (2\pi j_n) - j_n \log j_n + j_n},
\end{align*}
and we readily see that the exponent goes to $-\infty$ for $n\to +\infty$, as $\epsilon \log n$ dominates $\log \log n$ and $j_n\log j_n$ dominates the remaining terms. 

\vspace{2mm}

\item Arguing in a similarly way, we finally get 
\begin{align*}
\frac{c^{(n)}_{n-j_n-1}}{(n-1)!} \lesssim  \frac{\log^{j_n}(n)}{ j_n!} \sim e^{\log j_n + \log \log n - \frac 12 \log (2\pi j_n) - j_n \log j_n + j_n} \to 0,
\end{align*}
as $j_n \log j_n$ is the leading term for $j_n\gtrsim \log \log n$.

\vspace{2mm}

\item For fixed  $k\ge 2$,  we compute 
\begin{align*}
\frac{c^{(n)}_{n-j_n-1}}{(n-k)!} &\lesssim  (n-1)\cdot ... \cdot (n-k+1) \cdot \frac{\log^{j_n}(n)}{ j_n!} \\
&\sim e^{(k-1)\log n + \log j_n + \log \log n - \frac 12 \log (2\pi j_n) - j_n \log j_n + j_n} \to 0,
\end{align*}
as $j_n \log j_n$ is the leading term provided $j_n\gtrsim \log n$.\qedhere
\end{enumerate}
\end{proof}

%%%%%%%%%%%%%%%%%%%%%%%%%%%%%%%%
%%
%%
%%
%%
%%%%%%%%%%%%%%%%%%%%%%%%%%%%%%%%

\section{Lower bounds on the number of $\sbc$}\label{sec:lower-bounds}

This section is devoted to prove lower bounds on the number of $\sbc$ under the additional assumption
$$\text{{\bf(H2)\/} \hspace{45mm} All} \ \sbc \ \text{are non-degenerate}.\hspace{6cm}$$

 Such lower bounds follow from the estimates  provided in Section~\ref{sec:poincare} and the Morse inequalities.  We briefly recall that the Morse inequalities of a Morse function $f$ on a closed manifold $\mathcal M$ relates the Morse indices of the (non-degenerate) critical points to the Betti numbers of the manifold. Usually, these inequalities are expressed in terms of polynomial generating function. More precisely, we define the {\sc Morse polynomial\/} as 
\begin{equation}\label{eq:Morse-polynomial}
M(t)\=\sum_k \gamma_k t^k \qquad \textrm{where $\gamma_k$ is the number of critical points of $f$ having index $k$}
\end{equation}
and the {\sc Poincar\'e polynomial\/} as 
\begin{equation}\label{eq:Poincare-polynomial}
P(t)=\sum_k \beta_k t^k \qquad \textrm{where $\beta_k$ denotes the $k$-th Betti number of the manifold}
\end{equation}
i.e. the rank of the $k$-th homology group with real (or rational) coefficients. Then the Morse inequalities read
\begin{equation}\label{eq:Morse-poincare-rest-polynomial}
M(t)= P(t)+ (1+t) R(t)
\end{equation}
where $R(t)$ is  a polynomial with non-negative integer coefficients. So, as a direct consequence of the positivity of the coefficients of $R(t)$ one obtains that   $\beta_k$ is a lower bound for $\gamma_k$. 

\begin{rmk}
Even if the collision free configuration sphere $\widehat{\mathbb S}$ is not compact, we are still allowed to use the Morse inequalities to derive lower bounds 
on the number of $\sbc$. Indeed, as it is well-known, $U(q)\to +\infty$ as $q$ approaches the singular set $\Delta$. \qed
\end{rmk}
In Theorem~\ref{thm:main1} below we provide a lower bound on the number of planar $\sbc$ when all masses are equal. This lower bound depends on a spectral gap condition at any $\cccc$, which is 
independent of the chosen $\cccc$. Moreover, as a direct consequence of the Assumption~(H1), non-collinear $\sbc$ cannot be $\cc$.  

As in Section~\ref{sec:inertia-balanced}, $\eta_k < ... < \eta_1 < -U( \widehat q) < 0$
denote the eigenvalues of the matrix $M ^{-1}B(\widehat q)$.

\begin{thm}\label{thm:main1}
Suppose that $m_1=m_2=...=m_n$. Under the previous notation and Assumptions (H1) \& (H2), the following lower bounds on the number of planar $\sbc$ hold:
\begin{enumerate}
\item If $\displaystyle 1<s_1 < - \frac{\eta_1}{U(\widehat q)}$ there are at least 
\[
3n! - 2(n-1)!
\]
planar $\sbc$, of which at least 
\[
n! - 2(n-1)!
\]
are not collinear. %Moreover, for  $n\in \N$  sufficiently large, we get the following (improved) estimate 
%\[
%\big (n - (1+\gamma+\log n)\big )\cdot (n-1)!
%\]
%for non-collinear $\sbc$.
\vspace{1mm}
\item For $ - \dfrac{\eta_1}{U(\widehat q)}< s_1< - \dfrac{\eta_k}{U(\widehat q)}$ there are at least 
\[
4n! - 2(n-1)!
\]
planar $\sbc$ of which at least 
\[
2n! - 2(n-1)!
\]
are not collinear. Moreover, for $n\in \N$  sufficiently large, we get the following (improved) estimate 
\[
\big (3n - 2(1+n^\epsilon)\big )\cdot (n-1)!
\]
non-collinear $\sbc$ for some constant $\epsilon\=\epsilon(n)\to 0$ for $n\to +\infty$.
\vspace{1mm}
\item If $ - \dfrac{\eta_k}{U(\widehat q)} < s_1$ there are at least 
\[
5n! - 2(n-1)!-2
\]
planar $\sbc$, of which at least 
\[
3n! - 2(n-1)! -2
\]
are not collinear.
\end{enumerate}
\end{thm}
\begin{rmk}
For values of $s_1$ close to 1  we get worse lower bounds than in the other cases. This is not surprising, since in this case the Morse index  of $s_2-\csbc$  are precisely one less the Morse index of $s_1-\csbc$, and hence cancellations in homology might occur. 

By interpreting the parameter $s$ as a bifurcation parameter along the trivial branch of 1-CSBC, the scenarios provided in  Theorem~\ref{thm:main1} agree with the numbers of bifurcation instants as explained in \cite[Theorem 4.7]{AFP20}. Moreover, in the final section of \cite{AFP20} the authors produce for the three body problem with equal masses an explicit  description of this phenomenon as depicted in \cite[Figure 1]{AFP20}. In this respect, we also mention \cite[Theorem 4]{AD20} in which authors prove the existence of three smooth families of $\sbc$ as illustrated in Figure~2 of the aforementioned paper. 
\qed
\end{rmk}
\begin{proof}
All statements are  consequences of the estimates for the coefficients of the Poincar\'e polynomial proved in Section~\ref{sec:poincare} and the Morse inequalities, which may be written using Equation~\eqref{eq:Morse-poincare-rest-polynomial} as
\[
\sum_{i=0}^{n-1} \gamma_i z^i = p_n(z) + (1+z)R(z)= \sum_{i=0}^{n-1} c_i^{(n)}z^i+  (1+z)\sum_{i=0}^{n-1}r_iz^i,
\]
for some polynomial $R(z)$ having  non-negative integer coefficients $r_i$, where $\gamma_i$ denotes 
the number of critical points of $\widehat U$ having  Morse index $i$. Recall that the Morse index of $s_1-\csbc$ is $n-1$, as it follows from  Lemma~\ref{thm:cc-sbc-collinear}, independently of $s_1$. 
\begin{enumerate}
\item In virtue of Item 1  of  Proposition~\ref{lem:indicescollinear} for $j=0$, the assumption yields that the Morse index of $s_2-\csbc$ is $n-2$. In this case, we will simply ignore $s_2-\csbc$.
Applying the full Morse inequalities together with Proposition~\ref{lem:coeffofpoincarepolynomial} yields
\begin{align}
\gamma_{n-1} = (n-1)! + r_{n-2}+r_{n-1} \ge n! \quad &\Rightarrow \quad r_{n-2}+r_{n-1} \ge n!- (n-1)!
%\gamma_{n-2} = c_{n-2}^{(n)} + r_{n-3}+r_{n-2} \ge n! \quad &\Rightarrow \quad  r_{n-3}+r_{n-2} \ge n! - c_{n-2}^{(n)} \ge \frac{n!}{2}. \label{morse2}
 \end{align}
 Therefore, evaluating the Morse inequalities at $z=1$, we get
\begin{align*}
\sum_{i=0}^{n-1} \gamma_i = p_n(1) + 2R(1) & = n! + 2 \sum_{i=0}^{n-1} r_i\\
								&\ge n! + 2 ( r_{n-3}+r_{n-2}+r_{n-1})\\
								&\ge n! + 2 (r_{n-2}+r_{n-1})\\
								& \ge 3 n! - 2(n-1)!
								\end{align*}
Finally, we subtract $2n!$ which is the number of  $\csbc$ lying on the coordinate axes, which are the only possible $\csbc$ in virtue of Assumption (H1), thus obtaining the desired lower bound. 				

\item Under the given assumption, the Morse index $j$ of a $s_2-\csbc$ is strictly smaller than $n-2$ (in fact, at least $\eta_1$ moves into the positive line by adding $s_1 U(\widehat q)$). 
Applying again the full Morse inequalities together with Proposition~\ref{lem:coeffofpoincarepolynomial} we obtain
\begin{align}
\gamma_{n-1} = (n-1)! + r_{n-2}+r_{n-1} \ge n! \quad &\Rightarrow \quad r_{n-2}+r_{n-1} \ge n!- (n-1)! \label{morse3}\\
\gamma_{j} = c_{j}^{(n)} + r_{j-1}+r_{j} \ge n! \quad &\Rightarrow \quad  r_{j-1}+r_{j} \ge n! - c_{j}^{(n)} \ge \frac{n!}{2}, \label{morse4}
 \end{align}
 and hence 
 \[\sum_{i=0}^{n-1} \gamma_i \ge n! + 2 ( r_{j-1}+r_j + r_{n-2}+r_{n-1}) \ge 4n! - 2(n-1)!
 \]
 For $n\in \N$ sufficiently large enough we may replace~\eqref{morse4} with 
 \[
 r_{j-1}+r_{j} \ge n! - c_{j}^{(n)} \ge n! - n^\epsilon(n-1)!
 \]
 where $\epsilon$ is as given in Item 1 of Proposition~\ref{prop:asymptoticgrowth}. The claim follows.
\item In this case $s_2-\csbc$ are local minima of $\widehat U$ and hence
\[
\gamma_0 = c^{(n)}_0 + r_0+r_1 \ge n! \quad \Rightarrow \quad r_0+r_1 \ge n!-c_0^{(n)} = n!-1
\]
which yields the desired lower bound. \qedhere
\end{enumerate}
\end{proof}

In the next result we apply the asymptotic estimates for the coefficients of the Poincar\'e polynomial proved in Section \ref{sec:poincare} to improve the lower bounds on the number of  $\sbc$ for large values of $n$
in cases all masses are equal. The outcome will be that $s_j-\csbc$,  $j\ge 2$, asymptotically will - in many cases - give more contribution than the $s_1-\csbc$ to the count of $\sbc$. As already explained, the importance of such a result relies more in its qualitative  rather than in its quantitative aspect, since such improvements  will only be marginal.
We recall that $\alpha_j$ denotes the multiplicity of the eigenvalue $\eta_j$ of $M^{-1}B(\widehat q)$, where $\widehat q$ is any $\cccc$ of $n$ equal masses (cfr. Subsection~\ref{subsec:inertiaindices}
for further details).  

\begin{lem}
For $n\in \N$ large enough, we define $j_n\in \{1,...,n\}$ such that 
\[
\sum_{i=0}^{j_n-1} \alpha_i< \log n <\sum_{i=0}^{j_n} \alpha_i.
\]
Assuming that (H1) \& (H2) hold, for $\displaystyle - \frac{\eta_{j_n}}{U(\widehat q)} < s_1$ there exist at least 
$$5n! - 2(n-1)! - 2(n-h)!$$
planar  $\sbc$, for some $h\ge 1$, of which at least 
$$3n! - 2(n-1)! - 2(n-h)!$$
are not collinear.
\end{lem}
\begin{proof}
Under the assumptions, the Morse index $j$ of a $s_2-\csbc$ will be smaller than $n-\log n-1$. Using Item 3 of Proposition~\ref{prop:asymptoticgrowth}, we find an $h\ge 1$ such that 
\[
c^{(n)}_{j}\le (n-h)!
\]
and this yields by the Morse inequalities 
$$r_{j-1}+r_j \ge n! - (n-h)!$$
The claim follows by arguing as in the proof of Theorem~\ref{thm:main1}.
\end{proof}

We now drop the condition that all masses be equal and derive lower bounds on the number of $\sbc$ assuming (H1) and (H2). As already observed, in this case the jumps on the Morse index of $s_2-\csbc$
might occur at different values of $s_1$ for different orderings of the masses. Hence, a precise statement as in Theorem~\ref{thm:main1} is unfortunately no longer available. Imposing further conditions, 
one could still apply the asymptotic estimates proved in Section 3, thus obtaining better lower bounds. However, we refrain to do this in order to keep the exposition as 
reader friendly as possible.
Notice also that, by the pidgeonhole principle, at least $(n-1)!$ of the $s_2-\csbc$ must have the same Morse index; these will then contribute to the count of critical points provided the corresponding coefficient of the Poincar\'e polynomial be smaller than $(n-1)!$ (which is, in virtue of the discussion in Section~\ref{sec:poincare}, asymptotically very likely to be the case). 

\begin{thm}
\label{thm:main2}
Under Assumptions (H1) and (H2), if 
$$s_1> \max \Big \{\displaystyle - \frac{\eta_k}{U(\widehat q)} \ \Big |\ \widehat q \ \ \cccc \Big \},$$
then there are at least 
$$3n! - 2(n-1)! - 2$$
planar non-collinear $\sbc$. Otherwise, there are at least 
$$n!- 2(n-1)!$$
planar non-collinear $\sbc$. 
\end{thm}

\begin{proof}
 Under the given spectral gap assumption, we readily see that every $s_2-\csbc$ is
a local minimum of $\widehat U$. The claim follows as in Part 3 of Theorem~\ref{thm:main1}. The other claim is obtained using the Morse inequalities only taking into account the contribution given by $s_1-\csbc$.
The details are left to the reader. 
\end{proof}

For $d>2$ the Poincar\'e polynomial is given by
\[
p_n(t^{d-1}) = (1+ t^{d-1})(1+2t^{d-1}) \cdot \ldots \cdot (1+(n-1)t^{d-1}),
\]
and all coefficients in degree different from $j(d-1)$ vanish identically.  In the next proposition, we provide lower bounds for the number of $\sbc$ in case of equal masses
for some particular case: in the {\em worst possible case\/}, namely when the Morse indices of all $\csbc$ differ precisely by one, in an \textit{intermediate case}, namely when all Morse indices of $\csbc$ differ by at least two, and in the \textit{best possible case}, namely when the Morse indices of $\csbc$ are not integer multiples of $(d-1)$ 
for all but the $s_1$-collinear $\sbc$. 

\begin{prop}
For $m_1=...=m_n$, under Assumptions (H1) \& (H2), the following statements hold:
\begin{enumerate}
\item Suppose that, for every $j=1,...,d-1$, the Morse indices of $s_j-\csbc$ and $s_{j+1}-\csbc$ differ precisely by one. Then, there are at least 
\[
\left(d+\dfrac{1}{2}\right)n! - (n-1)!
\]
$\sbc$ of which at least
\[
\dfrac{n!}{2} - (n-1)!
\] 
are not collinear. For large values of $n\in\N$ the lower bound can be improved to 
\[
(n-(1+\gamma+\log n)) \cdot (n-1)!
\]
non collinear $\sbc$.
\vspace{1mm}
\item If, for every $j=1,\ldots,d-1$, the Morse indices of $s_j-\csbc$ and $s_{j+1}-\csbc$ differ at least by two, then there are at least 
\[
2 n! -2 (n-1)!
\]
non collinear $\sbc$.
\vspace{1mm}
\item Suppose that, for every $j\ge 1$, the Morse indices of $s_j-\csbc$  and $s_{j+1}-\csbc$ differ at least by two, and that for every $j>1$
the Morse index is not an integer multiple of $(d-1)$. Then,  there are at least 
\[
(d+1)n! -2 (n-1)!
\]
non collinear $\sbc$.
\end{enumerate}
\end{prop} 
\begin{proof}
$\ $ \\
\begin{enumerate}
\item In this case, the Morse indices of $\csbc$ are precisely 
\[
(d-1)(n-2), (d-1)(n-2)+1,...,(d-1)(n-1),
\]
so that the corresponding 
coefficients of the Poincar\'e polynomial vanishes only for the smallest resp. largest Morse index. The full Morse inequalities and Proposition~\ref{lem:coeffofpoincarepolynomial}  yield now for the coefficients of the remainder $R(z)$:
\begin{align*}
r_{(n-1)(d-1)-1} + r_{(n-1)(d-1)} &\ge n! - (n-1)!\\
r_{(n-1)(d-1)-j-1} + r_{(n-1)(d-1)-j} &\ge n! \quad  j =1,...,d-2,\\
r_{(n-2)(d-1)-1} + r_{(n-2)(d-1)} &\ge n! - c_{n-2}^{(n)} \ge \frac{n!}{2},
\end{align*}
and hence 
$$\sum_{j} \gamma_j \ge n! + (d-1)n! - (n-1)! + \dfrac{n!}{2} = \left(d+\dfrac12\right)n! -(n-1)!$$
The asymptotic estimate is obtained using again Item 3 of Proposition~\ref{lem:coeffofpoincarepolynomial}.

\item Denoting by $\mu_j$ the Morse index of $s_j-\csbc$ for $j>1$, from 
\begin{align*}
r_{(n-1)(d-1)-1} + r_{(n-1)(d-1)} &\ge n! - (n-1)!\\
r_{\mu_j-1} + r_{\mu_j} &\ge n! - c_{\mu_j}^{(n)} \ge \frac{n!}{2},
\end{align*}
we obtain 
$$\sum_{j} \gamma_j \ge n! + 2 (n! - (n-1)!) + (d-1)n! = (d+2)n! - 2(n-1)!$$
\item In this case we have 
$$r_{\mu_j-1} + r_{\mu_j} \ge n! - c_{\mu_j}^{(n)} = n! \quad \forall j>1$$
and hence 
$$\hspace{28mm} \ \  \sum_{j} \gamma_j \ge n! + 2 (n! - (n-1)!) + 2(d-1)n!=(2d+1)n! - 2(n-1)! \hspace{28mm} \ \qedhere$$
\end{enumerate}
\end{proof}

In the proposition above we applied Morse theory on the whole collision free configuration sphere in $\R^{nd}$,
 and as we can see, in two of the three cases we obtain lower bounds independent of the dimension $d$. 
 Implementing a strategy similar to the one used in Corollary~\ref{cor:continua}, we can instead give 
a general lower bound (i.e. independent of the Morse indices of the $\csbc$ as well as of the fact that the masses be equal or not) on the number of non-collinear $\sbc$ which grows quadratically in $d$. 

Before doing that in the next theorem, we shall notice that Part 1 of the proposition above holds also for non-equal masses provided 
all $s_j$ are sufficiently close to one. Indeed, also in this case all $s_j-\csbc$ have the same Morse index and the difference of Morse indices between $s_j-\csbc$ and $s_{j+1}-\csbc$ is always equal one. 
This shows that, at least for $s_j$ close to one, all $\csbc$ contribute to the count of $\sbc$, even if the masses are not all equal. Notice also that, ignoring $s_j-\csbc$ for $j\ge 2$ (as done in Theorem~\ref{thm:main1})
does not provide any non-trivial lower bound for $d>2$. 

\begin{thm}
For fixed $d \ge 2$, under Assumptions (H1) \& (H2), there are at least 
$$\frac{d(d-1)}{2} \cdot \big (n! - 2(n-1)!\big )$$
planar non-collinear $\sbc$. \qed
\label{thm:main3}
\end{thm}

\begin{proof}
For all $1\le i<j\le d$, the manifold $\mathcal P_{i,j}$ of configurations which are contained in the $(e_i,e_j)$-plane is invariant under the gradient flow of $\widehat U$. Therefore, we can do Morse theory on $\mathcal P_{i,j}$, which in turn is nothing else but the collision free configuration space 
in the plane. Notice indeed, that Assumptions (H1) and (H2) hold also for the restriction to $\mathcal P_{i,j}$. 
Therefore, using Theorem~\ref{thm:main2}, we obtain that there are at least 
$$n! - 2(n-1)!$$
non-collinear $\sbc$ in the $(e_i,e_j)$-plane. Furthermore, Assumption (H1) implies that the symmetry group of Equation~\eqref{eq:s-balanced-conf} is trivial, and hence
planar $\sbc$ contained in different planes can be counted separately. This together with the fact the number of different pairs $1\le i<j\le d$ is precisely $d(d-1)/2$ implies the claim.
\end{proof}

Also in this case, we could obtain better lower bounds by a case distinction and implementing the asymptotic growth estimates for the coefficients of the Poincar\'e polynomial. Again, we refrain
to do this to obtain better readability.  We finish this section noticing that information on the inertia indices of planar $\sbc$ might help to improve the above result by e.g. showing the existence of non-planar $\sbc$. This is indeed object 
of ongoing research. 

%%%%%%%%%%%%%%%%%%%%%%%%%%%%%%%%
%%
%%
%%
%%
%%%%%%%%%%%%%%%%%%%%%%%%%%%%%%%%

\section{The $45^\circ$ theorem for balanced configurations}\label{sec:45gradi}

The well-known {\sc $45^\circ$-Theorem\/} for collinear $\cc$ is  a ``global version'' of the fact that  collinear $\cc$ are an attractor for the projectivized gradient flow of $\widehat U$. This theorem, indeed, consists in providing an explicit neighborhood of the manifold of collinear configurations such that orbits of the gradient flow of $\widehat U$ starting off from such a neighborhood become more and more collinear. In this section we will extend the validity of the $45^\circ$-Theorem to $s_1-\csbc$ in $\R^3$. More precisely, we will show that there is a neighborhood  of the manifold of collinear configurations along the $e_1$-axis such that orbits of the gradient flow of $\widehat U$ starting off from such a neighborhood become  more and more collinear along the $e_1$-axis. 

Such a $45^\circ$-Theorem for $s_1-\csbc$ holds for any $S=\text{diag}\, (s_1,s_2,1)$, and actually even for every $d\ge 3$. However, to keep the discussion as elementary as possible we assume hereafter that $S=\diag (s,1,1)$ for some $s>1$, which is actually the case we will be interested in in Section~\ref{sec:homology}. 
We set
\[
S_M\= \begin{bmatrix}
	m_1 S & & \\
	& \ddots &\\
	&& m_n S 
\end{bmatrix}
\]
The gradient flow equation of $\widehat U$ is given by 
\begin{equation}\label{eq:gradient-flow}
	\dot q= S_M^{-1} \nabla U(q)+ U(q)q= \widetilde \nabla \widehat U(q)
\end{equation}

\begin{dfn}\label{def:angle-balanced}
The {\sc collinearity angle} of the configuration vector $q\in \widehat{\mathbb S}$  w.r.t. the $e_1$-direction is
\begin{equation}\label{eq:angolo-balanced}
\vartheta(q)\=\max_{i \neq j} \sphericalangle(q_i-q_j, e_1).
\end{equation}
Here $\sphericalangle(v,w)$ denotes the angle between the two vectors $v,w\in \R^3$. 
\end{dfn}

\begin{thm}[$45^\circ$-Theorem for $s_1-\csbc$] \label{thm_45-theorem} Consider the $S$-balanced configuration problem~\eqref{eq:s-balanced-conf} on $\R^3$ with 
$$S= \text{diag}\, (s,1,1), \qquad s>1.$$
Then the set 
$$\Big \{q\in\widehat S\ \Big |\ q \text{ is collinear along the } e_1\text{-axis}\Big \}$$
is an attractor for the gradient flow of $\widehat U$. More precisely, the
function pointwise defined in Equation~\eqref{eq:angolo-balanced} is a Lyapunov function on 
	\[
	\mathscr V\=\Set{q| \vartheta(q) \in (0, 45^\circ]}
	\]
	for $\widetilde \nabla \widehat U$. In particular there are no non-collinear $\sbc$ for $ \vartheta(q(t)) \in (0, 45^\circ]$.
\end{thm}

\begin{rmk}
A word by word generalization of Moeckel's $45^\circ$-Theorem (stating that the set of all configurations which are collinear along some line is an attractor for the 
gradient flow of $\widehat U$) cannot hold for $\sbc$, as $\csbc$ can be local minima of $\widehat U$ for suitable choices of $s>1$. In fact, a priori there is not even a 
reason why the set of collinear configurations  should be invariant under the gradient flow.
\qed
\end{rmk}

\begin{rmk}\label{rmk:complex}
$S$-balanced configuration in $\R^3$ with $S=\text{diag}\, (s,1,1)$ are of great interest since they produce relative equilibria of the $n$-body problem in $\R^4$ via ``complexification''. Notice that, 
unlike the case considered in Sections~\ref{sec:inertia-balanced}-\ref{sec:lower-bounds}  (where all eigenvalues of $S$ had multiplicity one), here there is no need to complexify the plane $\{0\}\times \R^2$, 
but rather we just have to complexify the (real) line corresponding to the eigenvalue $s$. This fact can be equivalently explained as follows: if we consider the $S$-balanced configurations problem in $\R^4$ with matrix $S=\text{diag}\, (s,s,1,1)$,
then restricting our attention only to $\sbc$ contained in $\{0\}\times \R^3$ would yield precisely the $S$-balanced problem in $\R^3$ with matrix $S=\text{diag}\, (s,1,1)$ considered above. 
\qed
\end{rmk}

\begin{proof}
We set $q= \trasp{(x,y,z)}$. Choose $(i,j)$ such that 
	$$\vartheta(q)= \sphericalangle(q_i-q_j, e_1),\qquad x_i-x_j>0.$$ 
	By construction, all other $q_k$ must lie in the intersection of the two cones of angle $\theta(q)$ with respect to $e_1$ at $q_i$ and $q_j$, which consists of the segment 
	joining $q_j$ with $q_i$ and the outer half-cones. Since by assumption $\theta(q)\le 45^\circ$, the two half-cones lie on opposite sides of the orthogonal bisector $B$ of $q_j$ and 
	$q_i$, which we recall is the plane orthogonal to $q_i-q_j$ through the middle point of the segment joining $q_j$ to $q_i$. By a simple minimality argument we can suppose without loss 
	of generality that none of the remaining $q_k$ lie in the segment joining $q_j$ to $q_i$. 
	
%By using the $\SO(2)$-symmetry provided by rotations in the $\widehat{yOz}$-plane, we can w.l.o.g. suppose that $q_i-q_j$ belongs to the 	$\widehat{xOz}$-plane. 
	Since for all $k$ we have
\begin{align*}
\dot q_k &= m_k^{-1} S^{-1}\nabla_k U(q)+U(q) q_k\\ &= 
m_k^{-1} \nabla_k U(q)+U(q) q_k- m_k^{-1} (1-s^{-1}) \cdot \text{pr}_1( \nabla_k U(q)),
\end{align*}
where 
$$\text{pr}_1(\nabla_k U(q)) = \left ( \sum_{\ell \neq k} m_\ell m_k \cdot \frac{x_\ell-x_k}{r_{\ell k}^3} \right )\cdot e_1$$
is the projection of $\nabla_k U(q)$ onto the $x$-axis, we get that 
\begin{align}\label{eq:*}
\dot q_i-\dot q_j&= m_i^{-1}\nabla_i U(q)- m_j^{-1} \nabla_j U(q)\\
&-
(1-s^{-1}) \left[m_i^{-1}\cdot \text{pr}_1 (\nabla_i U(q)) - m_j^{-1}\cdot \text{pr}_1(\nabla_j U(q))\right] +U(q) (q_i-q_j).\nonumber
\end{align}
Set $\alpha(t)\=\sphericalangle(q_i(t)-q_j(t), e_1)$. By the assumption $x_i-x_j>0$, we have
\[
\cos \alpha(t)=\left \langle \dfrac{q_i(t)-q_j(t)}{r_{ij}(t)}, e_1\right \rangle >0
\]
	Differentiating w.r.t. $t$  yields (the dependence on $t$ is hereafter dropped for notational convenience)
	\begin{align}\label{eq:**}
	-\dot \alpha \sin \alpha& = \left\langle \dfrac{\dot q_i-\dot q_j}{r_{ij}},  e_1\right\rangle - \left\langle\dfrac{\langle(q_i-q_j),  e_1\rangle}{r_{ij}^3}(q_i-q_j),(\dot q_i-\dot q_j)\right\rangle\\
	&= 
\left\langle	 \dfrac{\dot q_i-\dot q_j}{r_{ij}},\left[e_1- \dfrac{\langle(q_i-q_j), e_1\rangle}{r_{ij}^2}(q_i-q_j)\right]\right\rangle\nonumber  \\
&= \left\langle \dfrac{\dot q_i-\dot q_j}{r_{ij}}, e_1^B \right\rangle\nonumber 
	\end{align}
	where $e_1^B$ denotes the orthogonal projection of $e_1$ on the orthogonal bisector $B$ of $q_i$ and $q_j$. Observe that 
	$$\langle e_1,e_1^B \rangle >0.$$
	%\[
	%\|e_1^\perp\|= \sin \alpha\quad  \textrm{ and }\quad  e_1^\perp= \sin \alpha(\sin \alpha, 0, -\cos \alpha)
	%\]
	By using Equations~\eqref{eq:*} and \eqref{eq:**} and by taking into account that $\langle q_i-q_j, e_1^B\rangle =0$, we get  that
	\begin{align}\label{eq:***}
		-\dot \alpha \sin \alpha &= \left\langle\dfrac{1}{r_{ij}}\big[m_i^{-1}\nabla_i U(q)- m_j^{-1}\nabla_j U(q) \big], e_1^B\right\rangle\\
							 &- \dfrac{1-s^{-1}}{r_{ij}}\left\langle m_i^{-1}\cdot \text{pr}_1 (\nabla_i U(q))- m_j^{-1}\cdot \text{pr}_1(\nabla_j U(q)), e_1^B\right\rangle \nonumber \\
&=\dfrac{1}{r_{ij}}\left(\sum_{k \neq i}m_k \dfrac{q_k-q_i}{r_{ik}^3} e_1^B - \sum_{k \neq j}m_k\dfrac{q_k-q_j}{r_{jk}^3}e_1^B\right)\\ 
&+ \dfrac{1-s^{-1}}{r_{ij}} \left(\sum_{k \neq j}m_k \dfrac{x_k-x_j}{r_{jk}^3} - \sum_{k \neq i}m_k\dfrac{x_k-x_i}{r_{ik}^3}\right) \langle e_1,e_1^B\rangle. \nonumber
	\end{align} 
	 By arguing precisely as in \cite[P. 62-63]{Moe94}, we get that that the first summand in the (RHS) of Equation~\eqref{eq:***} is non-negative (actually, positive provided $q$ is not collinear along the line 
	 determined by $q_j$ and $q_i$). The last summand in the RHS of Equation~\eqref{eq:***} is the product of the positive number $(1-s^{-1})r_{ij}^{-1}\cdot \langle e_1,e_1^B\rangle$ with 
\begin{equation}\label{eq:ff2}
\sum_{k \neq j,i}m_k \left ( \dfrac{x_k-x_j}{r_{jk}^3}  - \dfrac{x_k-x_i}{r_{ik}^3}\right ) + m_i \dfrac{x_i-x_j}{r_{ji}^3} - m_j \dfrac{x_j-x_i}{r_{ij}^3}.
\end{equation}
The last two terms in Equation~\eqref{eq:ff2} are positive by the assumption $x_i-x_j>0$. Moreover, since by assumption none of the $q_k$ lies in the segment joining $q_j$ with $q_i$, for every $k\neq i,j$ we either have $x_k>x_i>x_j$ or $x_k<x_j<x_i$. If the former case holds, then $q_k$ lies in the outer cone at $q_i$ and hence the quantities
$$\dfrac{x_k-x_j}{r_{jk}^3} \quad \text{and} \quad \dfrac{x_k-x_i}{r_{ik}^3}$$ 
are both positive and the first one is larger than the second one. If the latter case holds, then $q_k$ lies in the outer cone at $q_j$ and the quantities 
$$\dfrac{x_k-x_j}{r_{jk}^3} \quad \text{and} \quad \dfrac{x_k-x_i}{r_{ik}^3}$$ 
are both negative and the second one is in absolute value larger than the first one. Overall, this shows that the expression in Equation~\eqref{eq:ff2} (thus also the RHS of Equation~\eqref{eq:***}) is always positive. Therefore, by~\eqref{eq:***} we can infer that $-\dot \alpha \sin \alpha >0$, thus completing the proof. 
	\end{proof}
A direct consequence of the $45^\circ$-theorem for $s_1-\csbc$ is the following.
\begin{cor}\label{thm:collinear-attractor}
	The set 
	\[
	\mathbb  Y\= \Set{q=\begin{bmatrix}
		\begin{pmatrix}
			x_1\\0\\0
		\end{pmatrix}, \ldots ,\begin{pmatrix}
			x_n\\0\\0
		\end{pmatrix}\end{bmatrix} \in \R^{3n} | \sum_{i=1}^n m_ix_i =0,\  \|q\|_{S_M}^2=\dfrac{1}{s}}
	\]
	is an {\sc attractor\/} for the gradient flow of $\widehat  U$. 
	
	The set $\mathscr V$ is invariant under the gradient flow and the function $t \mapsto \vartheta \big(q(t)\big)$ is strictly decreasing along gradient flow lines, provided $\vartheta \big(q(0)\big) \in (0, 45^\circ]$. \qed
\end{cor}
\begin{rem}
Corollary~\ref{thm:collinear-attractor} allows us to remove the set $\mathbb Y$ from $\widehat{\mathbb S}$ before quotienting out by the (diagonal) $\SO(2)$-action on $\widehat{\mathbb S}$ induced 
by rotations in the $\widehat{yOz}$-plane, which we recall is the symmetry group of 
Equation~\eqref{eq:s-balanced-conf} in case $S=\text{diag}\, (s, 1,1)$. Since $\mathbb Y$ is precisely the singular  set of the $\SO(2)$-action, we obtain that 
\begin{equation}
\overline{\mathcal S}\= \Big ( \widehat{\mathbb S} \setminus \mathbb  Y\Big )/\SO(2)
\label{quotientmanifold}
\end{equation}
is a manifold. Such an argument is actually very similar to the one used by Merkel in \cite{Mer08} in the case of spatial $\cc$ (the main difference being that we 
remove only collinear configuration in the $e_1$-direction instead of the whole manifold of collinear configurations), unlike in Merkel's case however the lower 
bound that we will obtain is much larger than McCord's lower bound on the number of planar $\cc$ (cfr. Section~\ref{sec:homology}). 
 \qed
\end{rem}
The goal of the next section will be to compute the Poincar\'e polynomial of $\overline{\mathcal S}$. This will give us - via the ``complexification argument'' discussed in Remark~\ref{rmk:complex} - a lower bound on the number of $\sbc$ in $\R^4$ with matrix $\text{diag}\, (s,s,1,1)$, that will be then compared with the lower bounds provided by Theorems~\ref{thm:main1} and~\ref{thm:main2} via the reduction to {\bf{(H1)} }argument.

We finish this section stating the general $45^\circ$-theorem for $s_1-\csbc$. The proof can be obtained from the one of Theorem~\ref{thm_45-theorem} with some minor modifications and will be omitted. 
\begin{thm}[$45^\circ$-theorem for $s_1$-collinear $\sbc$]\label{thm:45gradi-generale}
For $d\ge 2$, let $s_1>s_2>...>s_{d-1}>s_d=1$ be positive real numbers and $\mu_1,\ldots,\mu_d\in \N$ be natural numbers. Consider the $S$-balanced configurations problem 
\eqref{eq:s-balanced-conf} on $\R^D$, $D:= \mu_1+\ldots+\mu_d\ge d$, with 
\[
S= \diag(\underbrace{s_1,...,s_1}_{\mu_1\text{-times}}, ..., \underbrace{s_{d-1},...,s_{d-1}}_{\mu_{d-1}\text{-times}}, \underbrace{1,...,1}_{\mu_d\text{-times}}).
\]
Then, the set 
\[
\Big \{ q \in \widehat {\mathbb S} \ \Big |\ q \ \text{is collinear along some line in } \R^{\mu_1}\times \{0\}\subset \R^D\Big\}
\]
is an attractor for the gradient flow of $\widehat U$. More precisely, the {\sc collinearity function} 
$$\theta (q) := \min_{L} \max_{i\neq j} \sphericalangle(q_i-q_j, L),$$
where $L$ runs all over the lines through the origin in $\R^{\mu_1}\times (0)$, is a Lyapounov function for the gradient vector field of $\widehat U$ on the set 
$\{q \ |\ \theta (q) \in (0,45^\circ]\}$.\qed
\end{thm}

\begin{rmk}
The singular set of the action of the symmetry group of Equation~\ref{eq:s-balanced-conf} does not consist of 
only collinear configurations as soon as $d\ge  4$. This implies that the quotient of $\widehat{\mathbb S}$ by the group action will not be a manifold even if the $45^\circ$-theorem is available in any dimension. Indeed, the whole singular set cannot be removed, in general. This is already the case for planar  $\cc$, which are in general not an attractor for the gradient flow.
 Besides that, the computations of the homology of the (singular) quotient are unaccessible for large values of $d$, and 
even if in the particular case considered in this section and in Section~\ref{sec:homology}, we succeed to obtain a quotient manifold and to compute (though with some effort)
its homology, in the count of $\sbc$ we still have to keep in mind that  planar $\sbc$ which are contained in the $\widehat{yOz}$-plane are actually $\cc$, whereas the non-collinear $\sbc$ provided by Theorems~\ref{thm:main1} and~\ref{thm:main2} cannot be $\cc$.  

All these facts put on evidence, once again, the importance of the reduction to {\bf{(H1)}} argument. 
\qed
\end{rmk}

%%%%%%%%%%%%%%%%%%%%%%%%%%%%%%%%
%%
%%
%%
%%
%%%%%%%%%%%%%%%%%%%%%%%%%%%%%%%%

\section{A lower bound on the number of $\sbc$ in $\R^4$ \'a la McCord}\label{sec:homology}

In this last section we compute the homology of the manifold $\overline{\mathcal S}$ defined in~\eqref{quotientmanifold} and, by means of this information, give lower bounds 
on the number of $\sbc$ in $\R^4$ with matrix $\text{diag}\, (s,s,1,1)$ that will be then compared with the ones obtained in Theorem~\ref{thm:main1}.

\begin{thm}
\label{thm:main3}
For masses $m_1,...,m_n>0$ consider the $S$-balanced configurations problem~\eqref{eq:s-balanced-conf} on $\R^3$ with $S=\, \text{diag}\, (s,1,1)$, $s>1$. Under Assumption (H2), there exist at least 
\begin{equation}
\label{eq:lowerboundmccord}
n! \left (\sum_{j=2}^{n-1} \frac 1j + \frac 2n\right )
\end{equation}
$SO(2)$-families\footnote{As the problem is $SO(2)$-invariant, $\sbc$ always come in $SO(2)$-families.} of $\sbc$. As a corollary, for the corresponding $S$-balanced configurations problem~\eqref{eq:s-balanced-conf} on $\R^4$ with $S=\, \text{diag}\, (s,s,1,1)$ there are at least the same number as above of
$SO(2)\times SO(2)$-families of $\sbc$. 
\end{thm}

\begin{rem}\label{rmk:comparison}
The lower bound provided by Theorem~\ref{thm:main3} is much bigger than the ones in Theorems~\ref{thm:main1} and~\ref{thm:main2}, as it asymptotically grows like $n! \log n$. Moreover, unlike Theorems~\ref{thm:main1} and~\ref{thm:main2}, such a lower bound is independent of any spectral gap condition for $s$. The drawback is that we cannot a priori exclude that all $\sbc$ we find are 
in fact $\cc$ contained in $\{0\}\times \R^2$. However, Theorem~\ref{thm:main3} would be interesting even if this were the case, as the estimate provided by~\eqref{eq:lowerboundmccord} is larger than McCord's 
estimate \cite{McCord96} by a factor 2. \qed 
\end{rem}

The proof of Theorem~\ref{thm:main3} will take up the rest of the section. The two major ingredients are:
 \begin{enumerate}
	\item Corollary~\ref{thm:collinear-attractor}: The manifold $\mathbb Y$ is an attractor for the gradient flow of $\widehat U$, and the neighborhood $\mathscr V$ of $\mathbb Y$ 
	does not contain $\sbc$ other than the $s_1$-collinear $\sbc$. 
	\item Shub's lemma for $\sbc$\footnote{The proof of the original Shub's lemma carries over word by word to $\sbc$.}: $\sbc$ cannot accumulate on the singular set $\Delta$. 
	In particular, there exists an invariant neighborhood of $\Delta$ containing no $\sbc$.
\end{enumerate}

We set  $ \widehat{\mathbb Y} \= \mathbb Y \setminus \Delta$ and assume hereafter that $n \ge 4$, as for $n=3$ some arguments need some slight modification.  	
We start by collecting some well-known facts. The first two are direct consequence of the fact that both $\mathbb S$ and $\mathbb Y$ are homological spheres:
$$H_*(\mathbb S)=\left \{\begin{array}{r} \R  \quad \textrm{if}\  *=0, \ 3n-4,\\ 0  \qquad \qquad \! \textrm{otherwise},
\end{array}\right . \qquad   H_*(\mathbb Y)=\left \{\begin{array}{r} \R\quad  \text{if}\   *=0, \ n-2, \\ 0 \ \quad \qquad \textrm{otherwise}. \end{array}\right .$$
It is also readily seen that $\Delta$ is a codimensional-3-submanifold of $\mathbb S$. In particular, the set
\[
\mathbb S\setminus \big(\Delta \cup \mathbb Y\big)
\]
is simply connected. The set $\widehat{\mathbb Y}$ consists of $n!$ connected components (corresponding to the orderings of $n$ distinct points on the line), each of 
which is topologically a $(n-2)$-dimensional disk. In particular 
$$H_*(\widehat{\mathbb Y})=\left \{\begin{array}{r}
\R^{n!} \qquad \text{if}\ *=0, \\
0 \qquad \textrm{otherwise}.	
\end{array}\right .
$$
By the universal coefficients theorem and Lemma~\ref{thm_:poincare-polynomial} we also have
\[
\widetilde H_*(\widehat{\mathbb S}) \cong \widetilde H^*(\widehat{\mathbb S}) =\left \{
\begin{array}{r}
	\R^{c_j} \quad \text{if}\  *=2j,  \ \  j=1, \ldots, n-1,\\
	0 \qquad \qquad \qquad \qquad \ \ \ \ \textrm{otherwise},
\end{array}\right.
\]
where $c_j$ is the $j$-th coefficient of the polynomial $p(t)=(1+t)\cdot \ldots \cdot \big(1+(n-1)t\big)$ and $\widetilde H_*$ (resp. $ \widetilde H^*$) denotes the reduced homology (resp. cohomology). Therefore, by Alexander duality,
\[
\widetilde H_*(\Delta)\cong \widetilde H^{3n-5-*}(\widehat{\mathbb S})= \left \{ \begin{array}{r}
 \R^{c_j} \quad \text{if}\ *= 3n-5-2j, \ \   1 \le j \le n-1,	\\
 0 \qquad \qquad \qquad  \qquad \qquad \qquad \ \, \textrm{otherwise}.
 \end{array}\right .
\]
In particular, $ H_*(\Delta)$ vanishes in all degrees smaller than $n-4$ (except $*=0$), and 
\begin{align*}
H_{n-3}(\Delta)&\cong \R^{c_{n-1}}= \R^{(n-1)!},\\
 H_{n-1}(\Delta)&\cong \R^{c_{n-2}}= \R^{(n-1)!\sum_{j=1}^{n-1}\frac{1}{j}},\\
 & ...\\ 
H_{3n-7}(\Delta)&\cong \R^{c_{1}}= \R^{\frac{n(n+1)}{2}}.
\end{align*}

%%%%%%%%%%%%%%%%%%%%%%%%%%%%%%%%
%%
%%
%%
%%
%%%%%%%%%%%%%%%%%%%%%%%%%%%%%%%%

\subsection{Homology of some intermediate manifolds}\label{subsec:coll-mnfd-delta}

The computation of the homology of $\overline{\mathcal S}= \big (\mathbb S\setminus (\Delta \cup \mathbb Y)\big ) /\SO(2)$ is divided in the following intermediate steps:  
\\

{\bf (Step 1)} We compute the homology of $\mathbb Y \cup \Delta$.
\\

{\bf (Step 2)} We compute the reduced homology of $\mathbb S\setminus (\mathbb Y \cup \Delta )$ using Alexander duality.
\\

{\bf (Step 3)} We conclude using the Gysin long exact sequence of the fibration 
	$$ 
\xymatrix@R=20pt@C-6pt{
&  S^1 \ar@{^{(}->}[r] & \mathbb S \setminus (\mathbb Y \cup \Delta)\ar[d] & \\ & & \overline{\mathcal S}& 
}
$$

In this subsection we will discuss Steps 1 and 2, leaving Step 3 for the next subsection. 
Thus, we start computing the homology of $\mathbb Y \cup \Delta$ by using the Mayer-Vietoris sequence and the homological properties of $\mathbb Y$, $\mathbb Y \cap \Delta$, and $\Delta$. Since  
$$\mathbb Y=(\mathbb Y \setminus \Delta )\cup (\mathbb Y \cap \Delta)= \widehat{\mathbb Y}\cup (\mathbb Y \cap \Delta),$$ 
denoting by $\widehat \Delta$ a fattened open neighborhood of $\Delta$ in $\mathbb Y$, we get that 
\[
\mathbb Y= \widehat{\mathbb Y} \cup(\mathbb Y \cap \widehat \Delta)
\]
Each connected component of the intersection $\widehat{\mathbb Y} \cap (\mathbb Y \cap \widehat \Delta)$ is homotopy equivalent to $S^{n-3}$ (recall that the connected components of $\widehat{\mathbb Y}$ are homotopically $(n-2)$-dimensional disks). Since we have $n!$ connected components in total, we obtain 
\[
H_*\big(\widehat{\mathbb Y} \cap (\mathbb Y \cap \widehat \Delta\big)\cong 
\left \{\begin{array}{r}
	\R^{n!} \quad \text{if}\ *=0,n-3,\\
	0 \qquad \quad \ \ \textrm{otherwise}.
\end{array}\right.
\]
To compute the homology of $\mathbb Y \cap \Delta$, we start observing that $\Delta$ is an (ANR) in $\mathbb Y$ and hence
\[
H_*(\mathbb Y \cap \Delta)\cong H_*(\mathbb Y \cap \widehat \Delta).
\]
By using Mayer-Vietoris we obtain the following exact sequence
\[
\xymatrix@R=10pt@C-6pt{
  \cdots \ar[r] &
  {H}_{k+1}(\mathbb Y) \ar[r] 
%  \ar@{=}[d] 
&
  {H}_k(\widehat{\mathbb Y}\cap(\mathbb Y \cap \widehat \Delta))  \ar[r]&
  {H}_k(\widehat{\mathbb Y})\oplus {H}_k(\mathbb Y \cap \widehat \Delta)\ar[r]
%  \ar@{=}[d] 
&
  {H}_{k}(\mathbb Y) \ar[r] &
  \cdots 
%  \ar[r] &
%  {H}_0(X/A) \ar[r] &
%  0 \\
%  & B & & C
}
\]
In particular, using all available information for $k=0$ we obtain
$$\xymatrix@R=10pt@C-6pt{
  \cdots \ar[r] &  0 \ar[r]& \R^{n!}\ar[r]& \R^{n!}\oplus {H}_0(\mathbb Y \cap \widehat \Delta)\ar[r]& \R\ar[r]&0
}$$
which yields $H_0(\mathbb Y \cap \widehat \Delta) \cong \R$, whereas the long exact sequence for $k=1,...,n-3$ implies
$$H_j(\mathbb Y \cap \widehat \Delta)\cong 0,\quad \forall \ j=1,...,n-4.$$ 
Finally, the long exact sequence for $k=n-2$ reads
$$\xymatrix@R=10pt@C-6pt{
  \cdots \ar[r]& 0\ar[r]& 0\oplus {H}_{n-2}(\mathbb Y \cap \widehat \Delta)\ar[r]& \R \ar[r]&  \R^{n!} \ar[r]& 0 \oplus {H}_{n-3}(\mathbb Y \cap \widehat \Delta)\ar[r]& 0\ar[r]& \cdots
}$$
Since $\mathbb Y\cap \widehat \Delta$ is homotopy equivalent to $\mathbb Y \cap \Delta$, which is a codimension 1 subset of $\mathbb Y$, 
we have in particular that $H_{n-2}(\mathbb Y \cap \widehat \Delta)=0$. Hence, the sequence above becomes 
\[
\xymatrix@R=10pt@C-6pt{
  0 \ar[r] &\R \ar[r] & \R^{n!} \ar[r] &0 \oplus H_{n-3}(\mathbb Y \cap \widehat \Delta) \ar[r]& 0
  }
  \]
  By exactness we conclude  that 
  \[
  H_{n-3}(\mathbb Y \cap \widehat \Delta) \cong \R^{n!-1}.
  \]
Summing up, we have proved that  
\[
H_*(\mathbb Y\cap \Delta)\cong H_*(\mathbb Y \cap \widehat \Delta)= \left \{ \begin{array}{r} \R  \qquad \qquad \ \,  \text{if}\ *=0, \\
	\R^{n!-1} \quad \text{if}\ *=n-3,\\
	0 \qquad \ \ \ \ \, \textrm{ otherwise}. \end{array} \right .
\]

We are now ready to compute the homology of $\mathbb Y \cup \Delta$. To do this we use the Mayer-Vietoris sequence associated with $\mathbb Y, \mathbb Y \cap \Delta$ and $\Delta$:
\[
\xymatrix@R=10pt@C-6pt{
\cdots \ar[r] & H_k(\mathbb Y \cap \Delta) \ar[r] &H_k(\mathbb Y) \oplus H_k(\Delta) \ar[r] & H_k(\mathbb Y \cup \Delta) \ar[r]&H_{k-1}(\mathbb Y \cap \Delta) \ar[r] &\cdots 
}
\]
We immediately see that  the long exact sequence above implies 
$$H_*(\mathbb Y\cup \Delta)= \{0\},\quad \text{for} \ *=1,...,n-4,$$
whereas for $*=n-1,...,3n-7$ we have
\[
H_*(\mathbb Y \cup \Delta)\cong H_*(\Delta).
\]
For $k=n-2$ we get 
%\begin{itemize}
%\item \[
%\xymatrix@R=10pt@C-6pt{
%\ar[r] &H_{n-2}(\mathbb Y\cap \Delta) \ar[r] \ar@{=}[d]&H_{n-2}(\mathbb Y)\oplus H_{n-2}(\Delta) \ar@{=}[d]\ar[r]& H_{n-2}(\mathbb Y \cup \Delta) \ar[r]& H_{n-3}(\mathbb Y \cap \Delta)\ar[r] \ar@{=}[d]&\\
%\ar[r]&0\ar[r] & \R\ar[r] \oplus 0& H_{n-2}(\mathbb Y \cup \Delta) \ar[r] & \R^{n!-1} \ar[r]&\\
%\ar[r]&  H_{n-3}(\mathbb Y)\oplus H_{n-3}(\Delta)\ar[r]\ar@{=}[d]&H_{n-3}(\mathbb Y \cup \Delta)\ar[r]&H_{n-4}(\mathbb Y \cap \Delta)\ar[r]\ar@{=}[d]&\cdots\\
%\ar[r]& 0 \oplus \R^{(n-1)!} \ar[r]& H_{n-3}(\mathbb Y \cup \Delta)\ar[r]& 0 \ar[r]& \cdots
%}
%\]
%In conclusion, we get 
\[
\xymatrix@R=50pt@C+8pt{
0\  \ar[r] &\R \ar[r] &  H_{n-2}(\mathbb Y \cup \Delta) \ar[r] & \R^{n!-1} \ar[r]^g & \R^{(n-1)!}\ar[r]^h& H_{n-3}(\mathbb Y \cup \Delta) \ar[r]& 0
}
\]
Since $\mathbb Y \cup \Delta$ is homotopy equivalent to the union of $\Delta$  together with $n!$ disks of dimension $(n-2)$, we get 
\[
H_{n-3}(\mathbb Y \cup \Delta)\cong H_{n-3}( \Delta)\cong \R^{(n-1)!}
\]
and hence in particular $h$ is an isomorphism. By exactness, this implies that $\mathrm{Im}\, g=\{0\}$. Therefore, the exact sequence 
can be rewritten as 
 \[
\xymatrix@R=50pt@C+8pt{
0\  \ar[r] &\R \ar[r] &  H_{n-2}(\mathbb Y \cup \Delta) \ar[r] & \R^{n!-1} \ar[r] & 0
}
\]
and this readily implies that 
 \[
 H_{n-2}(\mathbb Y \cup \Delta)\cong \R^{n!}.
 \]
Summarizing, we have proved that 
\[
H_*(\mathbb Y \cup \Delta)\cong \left \{ \begin{array}{r}
 \R \qquad \qquad \qquad \qquad \qquad \qquad \quad \ \ \text{if}\ * =0, \\
 \R^{n!} \qquad \qquad \qquad \qquad \ \ \ \, \, \qquad \text{if}\  *=n-2,\\
 \R^{c_j} \quad \text{if}\  *= 3n-5-2j, \ j=1, \ldots, n-1,\\
 0 \qquad \qquad \qquad \qquad \qquad \qquad \ \  \textrm{ otherwise}.	
 \end{array}\right.
\]
Using Alexander duality and the universal coefficients theorem, we get 
\[
\widetilde H_{*}\big(\mathbb S\setminus (\mathbb Y \cup \Delta)\big)\cong \widetilde H^{3n-5-*}\big(\mathbb S \setminus (\mathbb Y \cup \Delta)\big)\cong \widetilde H_*(\mathbb Y \cup \Delta)
\]
and hence finally
\[
H_*\big(\mathbb S \setminus \mathbb Y \cup \Delta)\big)\cong\left \{ \begin{array}{r}
 \R \qquad \qquad \qquad \qquad \quad \ \ \text{if}\ *=0,\\
 \R^{n!} \qquad \qquad \qquad \ \ \, \text{if}\ *=2n-3, \\
 \R^{c_j} \quad \text{if}\ *= 2j, \ j=1, \ldots, n-1,\\
 0 \qquad \qquad \qquad \qquad \ \, \, \textrm{ otherwise}.	
 \end{array}\right .
\]

%%%%%%%%%%%%%%%%%%%%%%%%%%%%%%%%
%%
%%
%%
%%
%%%%%%%%%%%%%%%%%%%%%%%%%%%%%%%%

\subsection{The homology of $\overline{\mathcal S}$}
Let us now consider the fibration 
\[
\xymatrix@R=20pt@C-6pt{
& \mathbb S^1 \ar@{^{(}->}[r] & \mathbb S \setminus (\mathbb Y \cup \Delta)\ar[d] & \\ & & \overline{\mathcal S}& 
}
\]
Since $\mathbb S \setminus (\mathbb Y \cup \Delta)\big)$ is simply connected, the exact sequence of homotopy groups
\[
\xymatrix@R=10pt@C-6pt{
\cdots \ar[r]&\ar[r] \pi_1\big(\mathbb S \setminus (\mathbb Y\cup \Delta)\big) \ar[r] & \pi_1(\overline{\mathcal S}) \ar[r] & \pi_0(\mathbb S^1)
}
\]
yields that $\pi_1(\overline{\mathcal S})=\{0\}$. Therefore, the fibration is homologically orientable. Using the information provided in Subsection~\ref{subsec:coll-mnfd-delta}, 
we can rewrite the Gysin long exact sequence of the fibration
\[
\xymatrix@R=20pt@C-6pt{
\cdots \ar[r] & H_{k+1}\big(\mathbb S \setminus (\mathbb Y \cup \Delta\big) \ar[r] & H_{k+1}(\overline{\mathcal S}) \ar[r] & H_{k-1}(\overline{\mathcal S}) \ar[r] &  H_{k}\big(\mathbb S \setminus (\mathbb Y \cup \Delta\big) \ar[r] &\cdots
}
\]
as follows:
\[
\xymatrix@R=10pt@C-12pt{
 ... \ar[r] & H_{2n-3}(\overline{\mathcal S}) \ar[r] &  {\underbrace{H_{2n-2}\big(\mathbb S \setminus (\mathbb Y \cup \Delta\big)}_{=0}} \ar[r] &H_{2n-2}(\overline{\mathcal S}) \ar[r] & \\
 & H_{2n-4}(\overline{\mathcal S}) \ar[r]& {\underbrace{H_{2n-3}\big(\mathbb S \setminus (\mathbb Y \cup \Delta\big)}_{=\R^{n!}}} \ar[r] &H_{2n-3}(\overline{\mathcal S})\ar[r] & H_{2n-5}(\overline{\mathcal S}) \ar[r] &\\
%  \ar[r] &H_{2n-3}(\overline{\mathcal S})\ar[r] & H_{2n-5}(\overline{\mathcal S}) \ar[r]& H_{2n-4}\big(\mathbb S \setminus (\mathbb Y \cup \Delta\big)\ar@{=}[d]\\
% &&&\R^{c_{n-2}}\\
  &  {\underbrace{H_{2n-4}\big(\mathbb S \setminus (\mathbb Y \cup \Delta\big)}_{=\R^{c_{n-2}}}} \ar[r] &H_{2n-4}(\overline{\mathcal S})\ar[r] & H_{2n-6}(\overline{\mathcal S}) \ar[r]& {\underbrace{H_{2n-5}\big(\mathbb S \setminus (\mathbb Y \cup \Delta\big)}_{=0}}\ar[r] & &\\
  & & \ldots & \\
  &H_{2}(\overline{\mathcal S})\ar[r]  &H_{0}(\overline{\mathcal S}) \ar[r]& {\underbrace{H_{1}\big(\mathbb S \setminus (\mathbb Y \cup \Delta\big)}_{=0}} \ar[r]& {\underbrace{H_{1}(\overline{\mathcal S})}_{=0}} \ar[r] & \\
 & 0 \ar[r]& {\underbrace{H_{0}\big(\mathbb S \setminus (\mathbb Y \cup \Delta\big)}_{=0}} \ar[r]&H_0(\overline{\mathcal S}) \ar[r] &0
%%   H_{0}\big(\mathbb S \setminus (\mathbb Y \cup \Delta\big)\ar@{=}[d]\\
%% &0&&\R\\
% \ar[r] &0&& 
}
\]
Therefore, $H_0(\overline{\mathcal S})\cong \R$ and  $H_1(\overline{\mathcal S})\cong \{0\}$. Reading the long exact sequence backwards we obtain 
\[
\xymatrix@R=8pt@C-6pt{
0= H_1(\overline{\mathcal S}) \ar[r] & {\underbrace{H_2\big(\mathbb S \setminus (\mathbb Y \cup \Delta\big)}_{=\R^{c_1}}} \ar[r] & H_2(\overline{\mathcal S})\ar[r]&\R \ar[r]& 0
}
\]
which implies $H_2(\overline{\mathcal S})\cong \R \oplus \R^{c_1}$. Similarly,
\[
\xymatrix@R=8pt@C-6pt{
0= H_3(\mathbb S \setminus (\mathbb Y \cup \Delta)) \ar[r]& H_3(\overline{\mathcal S}) \ar[r] & {\underbrace{H_1(\overline{\mathcal S})}_{=0}} \ar[r]& ...}
\]
implies $H_3(\overline{\mathcal S})=0$. Computing further yields
\[
\xymatrix@R=8pt@C-6pt{
0= H_3(\overline{\mathcal S}) \ar[r] &  H_4(\mathbb S \setminus (\mathbb Y \cup \Delta))\ar[r] &H_4(\overline{\mathcal S})\ar[r]&H_2(\overline{\mathcal S}) \ar[r] & 0
}
\]
implying that 
\[
H_4(\overline{\mathcal S})\cong H_2(\overline{\mathcal S})\oplus 
H_4(\mathbb S \setminus (\mathbb Y \cup \Delta))\cong \R \oplus \R^{c_1} \oplus \R^{c_2},
\]
and arguing recursively we obtain for all but the top degree terms the following
\begin{lem}\label{thm:homology-quotient}
For every $k=0, \ldots, n-3$ we have:
\begin{align*}
\hspace{36mm} H_{2k+1}(\overline{\mathcal S})& =\{0\}, \hspace{36mm} \\
\hspace{36mm}  H_{2k}(\overline{\mathcal S})& =H_0(\overline{\mathcal S})\oplus \bigoplus_{j=1}^k H_{2j}(\mathbb S \setminus (\mathbb Y \cup \Delta))  = \R^{\sum_{j=0}^k c_j}. \hspace{36mm}
\qed
 \end{align*}
\end{lem}

For the top degree terms we first observe that 
	\[
	\xymatrix@R=8pt@C-6pt{
	0= H_{2n-5}(\overline{\mathcal S}) \ar[r] & {\underbrace{H_{2n-4}(\mathbb S \setminus (\mathbb Y \cup \Delta))}_{=\R^{c_{n-2}}}} \ar[r] &H_{2n-4}(\overline{\mathcal S}) \ar[r]& H_{2n-6}(\overline{\mathcal S}) \ar[r] & 0
	}
	\]
	implies that 
	$$H_{2n-4}(\overline{\mathcal S}) \cong \R^{\sum_{j=0}^{n-2} c_j}.$$ 
Moreover, arguing recursively we deduce from 
\[
\xymatrix@R=8pt@C-6pt{
{\underbrace{H_{2n}(\mathbb S \setminus (\mathbb Y \cup \Delta))}_{=0}} \ar[r] & H_{2n}(\overline{\mathcal S})\ar[r] &  H_{2n-2}(\overline{\mathcal S})\ar[r] & {\underbrace{H_{2n-1}(\mathbb S \setminus (\mathbb Y \cup \Delta))}_{=0}}, \\
{\underbrace{H_{2n+1}(\mathbb S \setminus (\mathbb Y \cup \Delta)}_{=0}}) \ar[r] & H_{2n+1}(\overline{\mathcal S})\ar[r] &  H_{2n-1}(\overline{\mathcal S})\ar[r] & {\underbrace{H_{2n}(\mathbb S \setminus (\mathbb Y \cup \Delta))}_{=0}}},
\]
and the fact that $\overline{\mathcal S}$ is finite dimensional that 
\[
 H_{j}(\overline{\mathcal S})\cong  0, \quad \forall j \ge 2n-2,
 \]
as all such homology groups are isomorphic. Finally, from 

	\begin{equation*}\label{eq:*finale-gysin}
	\xymatrix@R=8pt@C-6pt{
	0 \ar[r] & H_{2n-3}(\overline{\mathcal S}) \ar[r] & 
	 H_{2n-2}(\mathbb S \setminus (\mathbb Y \cup \Delta)) \ar[r] & 
	H_{2n-2}(\overline{\mathcal S})=0}
	\end{equation*}
we obtain using Part 1 in Proposition~\ref{lem:coeffofpoincarepolynomial}
\[
H_{2n-3}(\overline{\mathcal S}) \cong \R^{n! - \sum_{j=0}^{n-2} c_j} = \R^{c_{n-1}}= \R^{(n-1)!}
\]
Summarizing all the previous computations, we finally get the following result.
\begin{thm}\label{thm:homology-of-sbar}
The homology of $\overline{\mathcal S}$ is given by 
\begin{equation}\label{eq:ultima-forse}
 \hspace{38mm} H_*(\overline{\mathcal S})\cong \left \{\begin{array}{r}
	\R^{\sum_{j=0}^k c_j } \quad \text{if}\ *=2k, \ \ k=0, \ldots, n-2, \\
	\R^{(n-1)!} \qquad  \qquad \qquad \ \ \ \ \  \text{if}\  *= 2n-3,\\
	0 \qquad  \qquad \qquad \qquad \qquad \ \ \ \  \textrm{ otherwise}.
\end{array}\right. \hspace{38mm} \qed
\end{equation} 
\end{thm}

A new lower bound on the number of critical points of $\widehat U$ assuming non-degeneracy is given now by the sum of the Betti numbers of $\overline{\mathcal S}$:
\begin{equation}
(n-1)! + \sum_{k=0}^{n-2} \sum_{j=0}^k c_j = (n-1)! + \sum_{j=0}^{n-2} c_j (n-1-j). 
\label{lb1}
\end{equation}
The next lemma, which appears in \cite{McCord96} without proof, will be useful to find a closed form for~\eqref{lb1}.

\begin{lem}
\label{lb3}
Denote by $\xi_j^{(n)}$ the $j$-th coefficient of the polynomial 
$$(1+2t) \cdot ... \cdot (1+(n-1)t).$$
Then, 
\begin{align*}
\sum_{j=0}^{n-2} \xi_j^{(n)}  = \frac{n!}{2}, \qquad 
\sum_{j=0}^{n-3} \xi_j^{(n)} (n-2-j) = \frac{n!}{2} h(n),
\end{align*}
where $\displaystyle h(n) := \sum_{j=3}^{n} \frac 1j$. 
\end{lem}
\begin{proof}
The first identity follows evaluating the polynomial at $t=1$. We now prove the second identity by induction over $n$. A straightforward computation shows that for $n=4$ both RHS and LHS are equal to 7. Suppose now that the claim be true for $n$. We want to show that 
\begin{equation}
\sum_{j=0}^{n-2} \xi_j^{(n+1)} (n-1-j) = \frac{(n+1)!}{2} h(n+1).
\label{lb2}
\end{equation}
As one readily sees, we have 
$$\xi_j^{(n+1)} = \xi_j^{(n)} + n \xi_{j-1}^{(n)}, \quad \forall j =0,...,n-1,$$
where we set $\xi_{-1}^{(n)} = \xi_{n-1}^{(n)}\=0$. Therefore, the (LHS) of Equation~\eqref{lb2} can be rewritten by using the first identity and the inductive assumption as
\begin{align*}
\sum_{j=0}^{n-2} \xi_j^{(n+1)} (n-1-j) &= \sum_{j=0}^{n-2} \big (\xi_j^{(n)} + n \xi_{j-1}^{(n)}\big ) (n-1-j)\\
								&= \sum_{j=0}^{n-2}\xi_j^{(n)} + \sum_{j=0}^{n-3}\xi_j^{(n)} (n-2-j) + n \sum_{j=0}^{n-2} \xi_{j-1}^{(n)} (n-1-j)\\
								&= \frac{n!}{2} + \frac{n!}{2} h(n) + n \sum_{j=0}^{n-3} \xi_{j}^{(n)} (n-2-j)\\
								&= \frac{n!}{2} + \frac{n!}{2} h(n) + n\frac{n!}{2} h(n)\\
								%&= \frac{n!}{2} + \frac{(n+1)!}{2} h(n)\\
								&=  \frac{(n+1)!}{2} h(n+1).
\end{align*}
This completes the proof.
\end{proof}

For notational convenience, we hereafter drop the superscript from $\xi_j^{(n)}$. By the very definition of the coefficients $c_j$, we see that 
$$c_j = \xi_j + \xi_{j-1},\quad \forall j=0,...,n-1,$$
where as above $\xi_{-1}=\xi_{n-1}:=0$. Therefore, using Lemma~\ref{lb3} we can rewrite~\eqref{lb1} as 
\begin{align}
(n-1)! + \sum_{j=0}^{n-2} c_j (n-1-j) &= (n-1)! + \sum_{j=0}^{n-2} \big (\xi_j + \xi_{j-1}\big ) (n-1-j)\nonumber \\
							&= (n-1)! + \sum_{j=0}^{n-2} \xi_j + \sum_{j=0}^{n-3} \xi_j(n-2-j) + \sum_{j=0}^{n-2} \xi_{j-1} (n-1-j)\nonumber \\
							&= (n-1)!+ \frac{n!}{2} + n! h(n)\nonumber \\
							&= n! \Big ( h(n) + \frac 12 + \frac 1n\Big ). \label{lb4}
\end{align}
This gives the desired lower bound on the number of critical points of $\widehat U$ on $\overline{\mathcal S}$ assuming non-degeneracy, thus completing the proof of Theorem~\ref{thm:main3}. 
As already observed in Remark~\ref{rmk:comparison}, such a set of critical points 
contains also $\cc$ in the $\widehat{yOz}$-plane, which are in virtue of \cite{McCord96} (taking into account also $\ccc$) at least 
$$\frac{n!}{2} \big (h(n) + 1).$$ 

%%%%%%%%%%%%%%%%%%%%%%%%%%%%%%%%
%%
%%
%%
%%
%%%%%%%%%%%%%%%%%%%%%%%%%%%%%%%%

\newpage

\appendix

\section{Proof of Lemma \ref{thm:integrale-iterato}}\label{sec:appendix}

In this section we prove Lemma \ref{thm:integrale-iterato}.

\begin{lem}
For every $n,j\in \N$ we have 
$$ \int_1^n \frac{1}{i_1}\cdot \int_{i_1}^n \frac{1}{i_2}\cdot ... \cdot \int_{i_{j-1}}^n \frac{1}{i_{j}}\ \  \diff i_{j}\cdot ... \cdot \diff i_1 = \frac{1}{j!}\log^{j}(n).$$

\end{lem}
\begin{proof}
We divide the proof in two steps: 

\vspace{2mm}

\textbf{Step 1.} We prove by induction over $j\in \N$ that:
\begin{align}
 &\int_1^n \frac{1}{i_1}\cdot \int_{i_1}^n \frac{1}{i_2}\cdot ... \cdot \int_{i_{j-1}}^n \frac{1}{i_{j}}\ \  \diff i_{j}\cdot ... \cdot \diff i_1 = a_j \cdot \log^{j}(n), \label{1}\\
 &\int_{i_0}^n \frac{1}{i_1} \cdot \int_{i_1}^n \frac{1}{i_2}\cdot  ... \cdot \int_{i_{j-1}}^n \frac{1}{i_{j}} \ \diff i_{j}\cdot ... \cdot \diff i_1 = \sum_{k=0}^j \frac{(-1)^k}{k!} a_{j-k} \log^{j-k}(n) \log^k (i_0), \label{2}
\end{align}
where the $a_j$'s satisfy the recursive relation 
\begin{equation}
\left \{ \begin{array}{l} a_0= a_1 =1,\\ a_{j+1}= \displaystyle \sum_{k=0}^j \frac{(-1)^k}{(k+1)!} a_{j-k}, \quad \forall j\ge 1.\end{array}\right .
\label{recursiverelation}
\end{equation}

The claim is trivial for $j=1$. Indeed 
$$\int_1^n \frac{1}{i_1} \, \diff i_1 = \log(n) = a_1 \cdot \log^1(n)$$
and 
$$\int_{i_0}^n \frac{1}{i_1}\, \diff i_1 = \log (n) - \log (i_0) = \frac{(-1)^0}{0!} a_{1-0} \log^{1-0}(n) + \frac{(-1)^1}{1!} a_{1-1} \log^{1-1}(n)\log^1(i_0).$$

Suppose now that the claim be true for $j$ and compute for $j+1$ using \eqref{2} for $j$ and the recursive relation \eqref{recursiverelation} defining the $a_j$'s:
\begin{align*}
\int_1^n \frac{1}{i_1} \, \cdot &\int_{i_1}^n \frac{1}{i_2} \cdot ... \cdot \int_{i_j}^n \frac{1}{i_{j+1}}\ \diff i_{j+1}\cdot ...\cdot \diff i_2\, \diff i_1 \\
&= \int_1^n \frac{1}{i_1} \cdot \Big (\sum_{k=0}^j \frac{(-1)^k}{k!} a_{j-k} \log^{j-k}(n) \log^k (i_1) \Big )\,  \diff i_1 \\
&= \Big [ \sum_{k=0}^j \frac{(-1)^k}{(k+1)!} a_{j-k} \log^{j-k}(n) \log^{k+1} (i_1)   \Big ]^n_1\\
&= \log^{j+1} (n) \cdot \sum_{k=0}^j \frac{(-1)^k}{(k+1)!} a_{j-k} \\
&= a_{j+1} \cdot \log^{j+1}(n).
\end{align*}
This shows \eqref{1} for $j+1$. Finally, we compute using \eqref{1} for $j+1$: 
\begin{align*}
\int_{i_0}^n \frac{1}{i_1} \, \cdot &\int_{i_1}^n \frac{1}{i_2}\cdot  ... \cdot \int_{i_{j}}^n \frac{1}{i_{j+1}} \ \diff i_{j+1}\cdot ... \cdot \diff i_1\\
				&= a_{j+1} \cdot \log^{j+1}(n) - \sum_{k=0}^j \frac{(-1)^k}{(k+1)!} a_{j-k} \log^{j-k}(n) \log^{k+1} (i_0) \\
				&= a_{j+1} \cdot \log^{j+1}(n) + \sum_{k=0}^j \frac{(-1)^{k+1}}{(k+1)!} a_{j-k} \log^{j-k}(n) \log^{k+1} (i_0)\\
				&= a_{j+1} \cdot \log^{j+1}(n) + \sum_{k=1}^{j+1} \frac{(-1)^{k}}{(k)!} a_{j+1-k} \log^{j+1-k}(n) \log^{k} (i_0)\\
				&= \sum_{k=0}^{j+1} \frac{(-1)^{k}}{(k)!} a_{j+1-k} \log^{j+1-k}(n) \log^{k} (i_0),
\end{align*}
thus showing \eqref{2} for $j+1$. This completes the proof of Step 1.

\vspace{2mm}

\textbf{Step 2.} We prove that $a_j=\frac{1}{j!}$ for every $j\in\N$ by induction over $j\in\N$.
Thus, assume that the claim be true up to $j$ and compute using the recursive relation:
\begin{align*}
a_{j+1} &= \sum_{k=0}^j \frac{(-1)^k}{(k+1)!} a_{j-k}\\
		&= \sum_{k=0}^j \dfrac{(-1)^k}{(k+1)!} \dfrac{1}{(j-k)!}\\
		&= \sum_{k=0}^j \frac{(-1)^k}{(j+1)!} \binom{j+1}{k+1}\\
		&= \dfrac{1}{(j+1)!} \sum_{k=0}^j (-1)^k \binom{j+1}{k+1}\\
		&= \dfrac{1}{(j+1)!} \sum_{k=1}^{j+1} (-1)^{k-1} \binom{j+1}{k}\\
		&= \dfrac{1}{(j+1)!} \Big (1+ \sum_{k=0}^{j+1} (-1)^{k-1} \binom{j+1}{k}\Big ) \\
		&= \dfrac{1}{(j+1)!},
\end{align*}
where the last equality follows from the binomial theorem.
\end{proof}

\vspace{1cm}
\noindent
\textsc{Dr. Luca Asselle}\\
Justus Liebig Universit\"at Gie\ss en\\
Arndtrstrasse 2\\
35392, Gie\ss en\\
Germany\\
%Website: \url{https://sites.google.com/site/lucaasselle/}\\
E-mail: $\mathrm{luca.asselle@math.uni}$-$\mathrm{giessen.de}$

\vspace{5mm}
\noindent
\textsc{Prof. Alessandro Portaluri}\\
Università  degli Studi di Torino\\
Largo Paolo Braccini, 2 \\
10095 Grugliasco, Torino\\
Italy\\
%Website: \url{https://sites.google.com/view/alessandro-portaluri}\\
E-mail: $\mathrm{alessandro.portaluri@unito.it}$

\end{document}